\documentclass[11pt, reqno]{article}

\usepackage[margin=1in]{geometry}
\usepackage{hyperref}

\usepackage{mathtools}
\usepackage{amssymb}
\usepackage{amsthm}
\usepackage{amsfonts}
\usepackage{mathrsfs}

\usepackage{graphicx}
\usepackage{caption}
\usepackage{subcaption}

\counterwithin{figure}{subsection}

\newtheorem{thm}{Theorem}
\newtheorem{prp}[thm]{Proposition}
\newtheorem{lem}[thm]{Lemma}
\newtheorem{cor}[thm]{Corollary}
\newtheorem{con}[thm]{Conjecture}

\theoremstyle{definition}
\newtheorem{dfn}[thm]{Definition}
\newtheorem{exm}[thm]{Example}
\newtheorem{rem}[thm]{Remark}
\newtheorem{prb}[thm]{Problem}

\newcommand{\N}{\mathbb{N}}

\newcommand{\R}{\mathbb{R}}
\newcommand{\dSdw}{\Delta}
\newcommand{\uSdw}{\rotatebox[origin=c]{180}{$\Delta$}}
\newcommand{\iso}{\cong}
\def\id{\operatorname{id}}
\def\Lvl{\operatorname{Lvl}}
\def\Span{\operatorname{Span}}
\def\Hilb{\operatorname{Hilb}}

\def\im{\operatorname{in}}
\def\IMS{\operatorname{IMS}}
\def\IMV{\operatorname{IMV}}
\def\IMI{\operatorname{IMI}}
\def\Starts{\operatorname{Starts}}
\def\start{\operatorname{start}}
\def\Block{\operatorname{Block}}
\def\scd{\operatorname{scd}}
\def\ICscd{\operatorname{ICscd}}
\def\Hasse{\operatorname{Hasse}}
\def\Star{\operatorname{Star}}
\def\Pow{\operatorname{Pow}}
\def\DVP{\operatorname{DVP}}
\def\Spider{\operatorname{Spider}} 
\title{Macaulay Posets and Rings}
\author{Nikola Kuzmanovski}
\date{\vspace{-5ex}}

\begin{document}
	
\maketitle
\begin{abstract}
	Macaulay posets are posets in which an analog of the Kruskal--Katona Theorem holds.
	Macaulay rings (also called Macaulay-Lex rings) are rings in which an analog of Macaulay's Theorem for lex ideals holds.
	The study of both of these objects started with Macaulay almost a century ago.
	Since then, 
	these two branches have developed separately over the past century,
	with the last link being the Clements--Lindström Theorem.
	
	For every ring that is the quotient of a polynomial ring by a homogeneous ideal we define the poset of monomials.
	Under certain conditions,
	we prove a Macaulay Correspondence Theorem,
	a ring is Macaulay if and only if its poset of monomials is Macaulay.
	Furthermore, the tensor product of rings corresponds to the Cartesian product of the posets of monomials.
	This allows us to transfer results between rings and posets.
	The Macaulay Correspondence Theorem generalizes a theorem by Shakin, 
	and allows ideals that are not monomial with orders that are different from the lexicographic one.
	By using this translation, 
	we give several answers to a problem posed by Mermin and Peeva,
	a positive answer to Hoefel's question about applying Macaulay poset theory to ring theory,
	and deduce several other results in both algebra and extremal combinatorics.
	
	A new proof of the Mermin--Murai Theorem on colored square-free rings is presented by using star posets.
	We extend the Mermin--Murai Theorem to rings that are not square free by using the spider Macaulay Theorem of Bezrukov and Elsässer.
	Using the Mermin--Peeva and Shakin results about adding a variable to a ring such that it remains Macaulay,
	we give an answer to a question posed by Bezrukov and Leck about taking the product of a Macaulay poset with a chain.
	Some results of Chong also give answers to the Bezrukov--Leck problem.
	All of these results have a common feature.
	They involve the tensor product of rings whose Hasse graphs of the poset of monomials are trees.
	We call such rings, tree rings.
	We give a classification of Macaulay rings that are the tensor product of a tree ring.
	Finally, we show that there are Macaulay rings that are not the tensor product of tree rings,
	and present the first examples of Macaulay rings that are not quotients by a monomial ideal and not quotients by a toric ideal.

\end{abstract} \tableofcontents{\tiny }
\section{Introduction}\label{Introduction}

Macaulay posets (Definition \ref{Macaulay_Posets_Definition}) are posets in which an analog of the Kruskal--Katona Theorem holds.
They have been extensively studied in the literature on discrete extremal problems, 
appear as topics in textbooks \cite{BelaBook, engel_1997, FranklBook, HarperBook},
and survey papers have been devoted to them \cite{Bezrukov2000,Bezrukov2004}.
In many situations, they are special cases of more general vertex isoperimetric problems on graphs \cite{Bezrukov2004, BezrukovSergeiL.2009ASPo, BollobásBéla1991Caii, engel_1997, HarperBook}.

The Macaulay property on posets provides several applications in pure mathematics and engineering.
If a poset is Macaulay with a rank-greedy order then a solution to the maximum weight ideal problem follows.
Edge-isoperimetric problems on graphs reduce to the maximum weight ideal problem,
whence the Macaulay property gives solutions to edge-isoperimetric problems as well.
These implications were observed by Bezrukov in \cite{BezrukovSergeiL.1999Oaei},
and have been included among other discrete extremal problems in \cite{engel_1997, HarperBook}.
Edge-isoperimetric problems give solutions to many other problems.
The survey \cite{BezrukovEdgeSurvey} and the book \cite{HarperBook} go over some,
including the wirelength problem, 
the bisection width and edge congestion problem, 
and graph partitioning problems.
It is interesting to note that solving the edge-isoperimetric problem on the Petersen graph \cite{Bezrukov2000AnEP} was motivated by application to parallel processing.
For an application in pure mathematics,
Daykin \cite{DaykinD.E1974EfK} showed that the Erdös--Ko--Rado Theorem follows from the Kruskal--Katona Theorem.

The study of Macaulay rings (Definition \ref{Macaulay_Rings_and_Ideals_Definition})
started almost a century ago with Macaulay \cite{MacaulayF.S.1927SPoE}.
Macaulay's Theorem says that for every homogeneous ideal in a polynomial ring over a field,
there exists a monomial lex ideal with the same Hilbert function.
There has been a lot of interest in generalizing Macaulay's Theorem to quotients of polynomial rings
\cite{GasharovVesselin2008Hfot, GasharovVesselin2011Hsam,MerminJeff2010Bnol, MerminJeffrey2006Li, MerminJeffrey2007Hfal, MURAISATOSHI2011Frol, ShakinD.2007Mi, ShakinDA2001SgoM}.
Lex ideals play an important role in Hartshorne’s proof that the Hilbert scheme is connected \cite{Hartshorne1966ConnectednessOT}.
Another very remarkable result is due to Bigatti \cite{Bigatti}, Hulett \cite{Hulett} and Pardue \cite{PardueKeith1996Dcog},
which states that lex ideals in a polynomial ring have the largest graded Betti numbers among all ideals with the same Hilbert function.

We have two main contributions to both of these analogs of Macaulay theorem stated in terms of posets and stated in terms of rings.
The first (Section \ref{Translating_Between_Posets_and_Rings}) is being able to translate Macaulay theorems between posets and rings.
The second (Section \ref{Quotients_by_Monomial_Ideals} and Section \ref{Quotients_by_Binomial_Ideals}) is applying these translation techniques to deduce new and old facts for both theories.
Section \ref{Advanced_Orders} is dedicated to giving general classes of orders that unify Macaulay theorems under a single guise.
A lot of questions come up based on the previous sections and we list these problems in Section \ref{Open_Problems}.

The main dictionary results translating between rings and posets are Theorem \ref{Macaulay_Correspondence_Theorem} and Theorem \ref{The_Cartesian_and_Tensor_Correspondence}.
Another useful tool for translating between posets and rings is Bezrukov's Dual Lemma \ref{Bezrukov_Dual_Lemma},
but this is not new.
These theorems are achieved by defining the poset of monomials (see Section \ref{The_Poset_Of_Monomials})
for every quotient of a polynomial ring by a homogeneous ideal.
The results that follow are deduced by applying these translation theorems.
These applications give many answers to a problem posed by Mermin and Peeva \ref{The_General_Mermin_Peeva_Problem}.
As a byproduct, we give a positive answer to a question of Hoefel \cite{Hoefel}.
Hoefel wanted to know if Macaulay poset theory can be applied to the theory of Macaulay rings.
The answer is an overwhelming yes,
and furthermore the implications go both ways.

We show that the Mermin--Murai Theorem \ref{MerminMurai} follows from the Macaulay problem on star posets,
which was settled two decades earlier.
It turns out that the Mermin--Murai Theorem holds for rings that are not square-free.
This is given in Theorem \ref{Bezrukov_Elsässer_Macaulay_Theorem}.
From a combinatorial viewpoint we give an answer (Corollary \ref{Bezrukov_Leck_Chain}) to a problem by Bezrukov and Leck,
which asks which Macaulay posets remain Macaulay after we take their cartesian product with a chain.
This is done by using a theorem independently proven by Mermin–Peeva and Shakin \ref{Mermin_Peeva_and_Shakin} that allows us to add a variable to a Macaulay ring, such that it remains Macaulay.
Chong's results in Section \ref{Products_With_Multiset_Lattices} also give answers to the Bezrukov--Leck problem.

All the above mentioned rings have something in common.
They can be decomposed into a tensor product of rings,
where the Hasse diagram of the poset of monomials of every individual ring is a tree.
In Theorem \ref{treeRingClassification} we obtain a classification of such rings,
where we use copies of one ring in the product.
Going in another direction, 
we show (Theorem \ref{Leck_Ring_Macaulay}) that there are Macaulay rings that are a tensor product, 
not involving any rings with this tree condition.

All the results up until now do not use the full power of the machinery developed in Section \ref{Translating_Between_Posets_and_Rings}.
In Section \ref{Quotients_by_Binomial_Ideals} we give the first examples of Macaulay rings that are not quotients by a monomial or toric ideal,
see Corollary \ref{Tori_Rings_Macaulay} and Corollary \ref{Diamond_Macaulay}.
These corollaries are obtained by using all of the ideas and results developed in Section \ref{Translating_Between_Posets_and_Rings}. \section{Translating Between Posets and Rings}\label{Translating_Between_Posets_and_Rings}

This section introduces all the tools needed to translate results between posets and rings.
The results obtained in this section are key to understanding and proving claims in Section \ref{Quotients_by_Monomial_Ideals}, Section \ref{Quotients_by_Binomial_Ideals}, and Section \ref{Open_Problems}.

\numberwithin{thm}{subsection}
\subsection{Partially Ordered Sets}\label{Posets}

This section introduces definitions and notation for partially ordered sets needed for the rest of the paper.
Most definitions here are based on Engel's book \cite{engel_1997}.
However, some notation is slightly different compared to \cite{engel_1997}, and some new concepts are introduced.

We define the set of natural numbers $\N$ to include $0$.
For $n\in \N$ we define $[n]=\{1,\dots, n\}$ and $[n]_0 = \{0,\dots , n-1\}$.
We also define $[\infty] = \N \setminus \{0\}$ and $[\infty]_0 = \N$.

A \textit{partially ordered set} (\textit{poset}) is a pair $\mathscr{P}=(S,\mathcal{O})$, where $S$ is a set and $\mathcal{O}$ is a partial order (a reflexive, antisymmetric and transitive relation) on $S$.
We will often abuse notation conflating $\mathscr{P}$ and $S$.
For $a,b\in \mathscr{P}$ we say $a \leq_{\mathscr{P}} b$ iff $(a,b)\in \mathcal{O}$.
We write $a <_{\mathscr{P}} b$ iff $a\leq_{\mathscr{P}} b$ and $a\neq b$.
Also, we say that $b$ \textit{covers} of $a$ iff $a < b$ and there is no $x\in \mathscr{P}$ such that $a < x < b$.
We say that $a$ is a \textit{minimum element (resp. maximum)} of $\mathscr{P}$ iff whenever we have $x\in\mathscr{P}$ such that $x\leq a$ (resp. $x\geq a$) then $x=a$.
The \textit{Hasse graph} of $\mathscr{P}$ is $\Hasse(\mathscr{P}) = (V,E)$,
where $V=\mathscr{P}$ and two vertices are connected by an edge if one covers the other.
For $A\subseteq S$ we denote the \textit{restriction} of $\mathcal{O}$ to $A$ by $\mathcal{O} |_{A}$.
Furthermore we call $(A, \mathcal{O} |_A)$ a \textit{subposet}.

Suppose that $S$ is a set.
A \textit{grid point/lattice point/multiset} on $S$ is a function $f:S \rightarrow \N$.
If $S$ is finite then we define the size of $f$ by $|f| = \sum_{x\in S} f(x)$,
and we say that $f$ has dimension $|S|$.
If $S= \{a_1,\dots, a_d\}$ and we consider $\ell_1,\dots, \ell_d\in \N \cup \{\infty\}$,
with $d\geq 1$,
then we define the \textit{set of all grid points/lattice points/multisets on} $S$ \textit{and lengths} $(\ell_1,\dots, \ell_d)$ by
\begin{align*}
	M_S(\ell_1,\dots, \ell_d) = \{f:S\rightarrow \N \bigm | f(a_1)< \ell_1,\dots, f(a_d)< \ell_d\}.
\end{align*}
Note that one can identify $f\in M_{[d]}(\ell_1,\dots, \ell_d)$ with $(f(1),\dots, f(d))\in \N^d$.

Suppose that $f,g$ are multisets on $S$ with $|S|=d\geq 1$ and consider $\ell_1,\dots, \ell_d\in \N \cup \{\infty\}$.
We say that $f$ is a  \textit{submultiset} of $g$ and write $f\subseteq g$ if and only if for all $x\in S$ we have $f(x) \leq g(x)$.
Then $\subseteq$ is a partial order on $M_S(\ell_1,\dots, \ell_d)$.
The \textit{lattice of multisets of dimension} $d$ \textit{on} $S$ \textit{and lengths} $(\ell_1,\dots, \ell_d)$ is the poset $\mathscr{M}_S(\ell_1,\dots, \ell_d) = (M_S(\ell_1,\dots, \ell_d), \subseteq).$

\begin{dfn}[Ranked Posets]\label{Ranked_Posets_Definition}
	Suppose that we have a poset $\mathscr{P}$ and a function $r:\mathscr{P} \rightarrow \N$.
	The function $r$ is called a \textit{rank function} on $\mathscr{P}$ if the following conditions hold:
	\begin{enumerate}
		\item There is a minimum element $x\in \mathscr{P}$ such that $r(x)=0$.
		\item Whenever we have $a,b\in \mathscr{P}$ such that $b$ covers $a$, then we must have $r(a)+1=r(b)$.
	\end{enumerate}
	We also call $\mathscr{P}$ a \textit{ranked poset} if such an $r$ exists.
	We define the \textit{rank of} $\mathscr{P}$ to be
	\begin{align*}
		r(\mathscr{P}) = \sup_{a\in \mathscr{P}} r(a).
	\end{align*}
\end{dfn}

\begin{exm}
	Any multiset lattice of the form $\mathscr{M}=\mathscr{M}_{[d]}(\ell_1,\dots, \ell_d)$ is ranked with $r:\mathscr{M} \rightarrow \N$ given by $r(f) = \sum_{x\in [d]}f(x)$ for all $f\in \mathscr{M}$.
	Furthermore, $r$ is the unique rank function on $\mathscr{M}$.
\end{exm}

\begin{dfn}[Cartesian Product of Posets]\label{Cartesian_Product_Of_Posets_Definition}
	Suppose that $d\in \N$ with $d\geq 1$ and consider posets $\mathscr{P}_i = (S_i, \mathcal{O}_i)$ for all $i\in [d]$.
	First, put $S = S_1\times \cdots \times S_d$.
	Next, we define a partial order $\mathcal{O}$ on $S$, such that
	for any $(x_1,\dots , x_d),(y_1, \dots y_d)\in S$ we have
	\begin{align*}
		(x_1,\dots, x_d) \leq_{\mathcal{O}} (y_1,\dots , y_d) \mbox{ if and only if for all } i\in [d] \mbox{ we have } x_i \leq_{\mathcal{O}_i} y_i.
	\end{align*}
	We write
	\begin{align*}
		\mathcal{O} = \mathcal{O}_1 \times \cdots \times \mathcal{O}_d.
	\end{align*}
	The poset $\mathscr{P} = (S, \mathcal{O})$ is called the \textit{Cartesian product of} $(\mathscr{P}_1,\dots ,\mathscr{P}_d)$ and we write
	\begin{align*}
		\mathscr{P} = \mathscr{P}_1\times \cdots \times \mathscr{P}_d.
	\end{align*}
	Next, suppose that $\mathscr{P}_1,\dots, \mathscr{P_d}$ have rank functions $r_1,\dots, r_d$ respectively.
	We define the \textit{product rank function} $r:\mathscr{P} \rightarrow \N$ such that $r((x_1,\dots, x_d)) = r_1(x_1)+\cdots + r_d(x_d)$ for all $(x_1,\dots, x_d)\in \mathscr{P}$.
\end{dfn}

Given posets $\mathscr{P}_1$ and $\mathscr{P}_2$, 
we write $\mathscr{P}_1 \iso \mathscr{P}_2$ to mean that $\mathscr{P}_1$ and $\mathscr{P}_2$ are isomorphic, 
i.e. there exists a bijection $\sigma : \mathscr{P}_1 \rightarrow \mathscr{P}_2$ such that for all $a, b \in \mathscr{P}_1$, 
$a \leq b$ if and only if $\sigma(a) \leq \sigma(b)$.

Suppose that we have a poset $\mathscr{P} = (S, \mathcal{O})$.
We say that $\mathcal{O}$ is a \textit{total order} on $S$ if for all $a,b\in S$,
we have $a\leq b$ or $b\leq a$.
If $\mathcal{O}$ is a total order then we call $\mathscr{P}$ a \textit{totally ordered set} or \textit{toset} for brevity.
All tosets considered in this paper will be isomorphic to $[n]$ with the standard order, where $n\in \N \cup \{\infty\}$.

\begin{dfn}[Indices, Elements and Intervals]\label{Indices_Elements_And_Intervals_Definition}
	Let $\mathscr{T}$ be a toset.
	Then there is a bijective function $\sigma: \mathscr{T} \rightarrow [n]$ such that for any $a,b\in \mathscr{T}$ we have $a\leq b$ iff $\sigma(a) \leq \sigma(b)$.
	Let $x,y\in \mathscr{T}$ with $x \leq y$ and any $p,q\in [n]$ with $p\leq q$.
	The \textit{index of} $x$ is $\mathscr{T}(x) = \sigma(x)$.
	The \textit{element of} $p$ is $\mathscr{T}^{-1}(p) = \sigma^{-1}(p)$.
	The \textit{closed (resp. half open )interval between} $x$ \textit{and} $y$ is $\mathscr{T}[x,y] = \{a\in \mathscr{T} \bigm | x\leq a \leq y\}$ (resp. $\mathscr{T}[x,y) = \{a\in \mathscr{T} \bigm | x\leq a < y\}$).
	The \textit{closed (resp. half open )interval between} $p$ \textit{and} $q$ is $\mathscr{T}[p,q] = \mathscr{T}[\mathscr{T}^{-1}(p), \mathscr{T}^{-1}(q)]$ (resp. $\mathscr{T}[p,q) = \mathscr{T}[\mathscr{T}^{-1}(p), \mathscr{T}^{-1}(q))$).
	We call $\mathscr{T}[q] = \mathscr{T}[1,q]$ the \textit{initial segment of size $q$ in} $\mathscr{T}$.
	If the total order of $\mathscr{T}$ is called $\mathcal{O}$,
	we call $\mathscr{T}[q]$ the \textit{initial segment of size $q$ of} $\mathcal{O}$.
	For a finite set $A\subseteq \mathscr{T}$ we will often write $\mathscr{T}[A]$ to mean $\mathscr{T}[|A|]$.
\end{dfn}

\begin{prp}[Decomposition of Multiset Lattices]\label{Decomposition_Of_Multiset_Latices}
	Suppose that $d\in \N$ with $d\geq 1$
	Consider tosets $\mathscr{T}_1,\dots, \mathscr{T}_d$ and put $\mathscr{T}=\mathscr{T}_1\times \cdots \times \mathscr{T}_d$.
	One has,
	\begin{align*}
		\mathscr{T} \iso \mathscr{M}_{[d]}(|\mathscr{T}_1|,\dots, |\mathscr{T}_d|),
	\end{align*}
	where the isomorphism sends $(x_1,\dots, x_d)\in \mathscr{T}$ to $(\mathscr{T}(x_1)-1,\dots, \mathscr{T}(x_d)-1)$.
\end{prp}

\begin{dfn}[Cube Diagrams]\label{Cube_Diagrams_Definition}
	Suppose that $\mathscr{T}_1,\dots, \mathscr{T}_d$ are tosets and put $\mathscr{T}=\mathscr{T}_1\times \dots \times \mathscr{T}_d$, 
	where $d\in \N$ is nonzero.
	Take $x=(x_1,\dots, x_d)\in \mathscr{T}$.
	Then we can consider $(\mathscr{T}_1(x_1) - 1,\dots, \mathscr{T}_d(x_d) - 1)\in \R^d$.
	For the point $(\mathscr{T}_1(x_1) - 1,\dots, \mathscr{T}_d(x_d) - 1)$ we consider the \textit{upwards cube of side length $1$ in} $\R^d$, with lower left corner at the given point,
	\begin{align*}
		\prod_{i=1}^d [\mathscr{T}_i(x_i),\mathscr{T}_i(x_i)+1].
	\end{align*}
	Thus, for each $x=(x_1,\dots, x_d)\in \mathscr{T}$ there is a unique upwards cube in $\R^d$.
	For a set $A\subseteq \mathscr{T}$ we call the collection of all upwards cubes from elements in $A$ the \textit{cube diagram of} $A$ in $\mathscr{T}$.
\end{dfn}

\begin{figure}
	\centering
	\begin{subfigure}[t]{0.23\textwidth}
		\includegraphics[width=\textwidth]{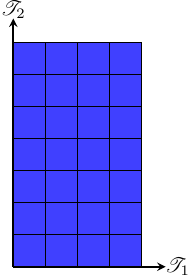}
		\caption{$\emptyset$.}
	\end{subfigure}
	\hfill
	\begin{subfigure}[t]{0.23\textwidth}
		\includegraphics[width=\textwidth]{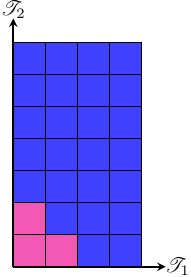}
		\caption{$\{(0,0),(0,1),(1,0)\}$.}
	\end{subfigure}
	\hfill
	\begin{subfigure}[t]{0.23\textwidth}
		\includegraphics[width=\textwidth]{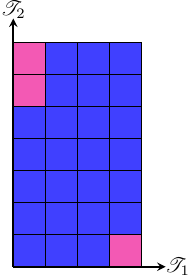}
		\caption{$\{(0,5),(0,6),(3,0)\}$.}
	\end{subfigure}
	\hfill
	\begin{subfigure}[t]{0.23\textwidth}
		\includegraphics[width=\textwidth]{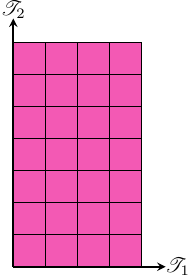}
		\caption{$\mathscr{M}_{[2]}(4,7)$.}
	\end{subfigure}
	
	\caption{Cube diagrams for some subsets in $\mathscr{M}_{[2]}(4,7)$.}
	\label{Cube_Diagrams_4x7}
\end{figure}

\begin{figure}
	\centering
	\begin{subfigure}[t]{0.75\textwidth}
		\includegraphics[width=\textwidth]{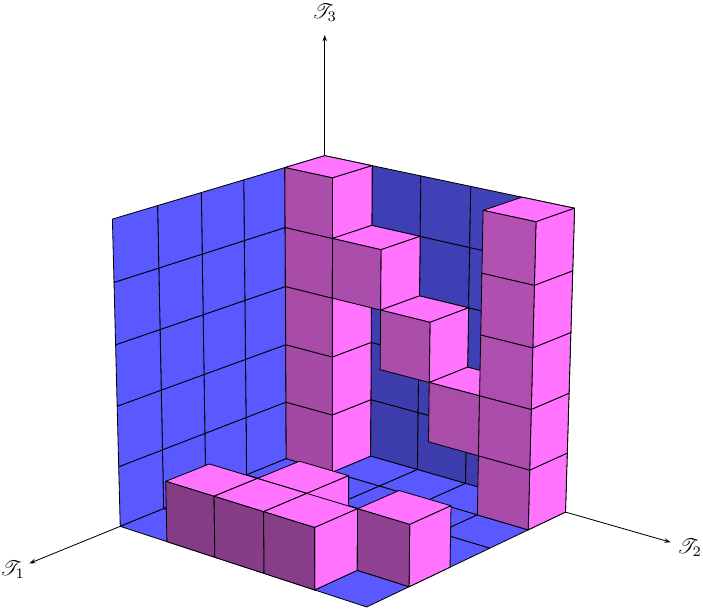}
	\end{subfigure}
	
	\begin{subfigure}[t]{0.46\textwidth}
		\includegraphics[width=\textwidth]{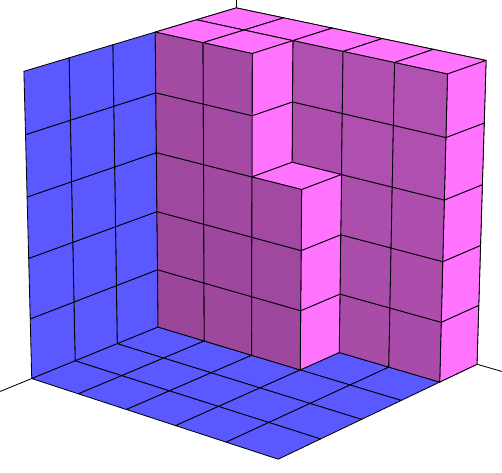}
	\end{subfigure}
	\hfill
	\begin{subfigure}[t]{0.46\textwidth}
		\includegraphics[width=\textwidth]{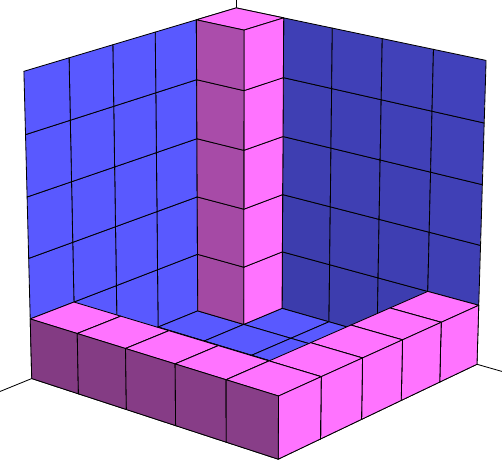}
	\end{subfigure}
	
	\caption{Cube diagrams of some sets in $\mathscr{M}_{[3]}(5,5,5)$.}\label{Cube_Diagrams_5x5x5}
\end{figure}

\begin{exm}
	Some cube diagrams of some sets in $\mathscr{M}_{[2]}(4,7)$ are shown in Figure \ref{Cube_Diagrams_4x7}.
	Some cube diagrams of sets in $\mathscr{M}_{[3]}(5,5,5)$ are shown in Figure \ref{Cube_Diagrams_5x5x5}.
\end{exm}

\begin{exm}
	$\mathscr{M}_{[d]}(\ell_1,\dots, \ell_d)$ has a cube diagram given by all the upward cubes in $[0,\ell_1]\times \cdots \times [0,\ell_d]$ hyperrectangle in $\R^d$.
\end{exm}

\begin{dfn}[Lexicographic Order]\label{Lexicographic_Order_Definition}
	Suppose that we have tosets $\mathscr{T}_1, \dots , \mathscr{T}_d$ for $d\in \N$ with $d\geq 1$.
	Consider $\mathscr{T} = \mathscr{T}_1\times \cdots \times \mathscr{T}_d$ and $x=(x_1,\dots, x_d),y=(y_1,\dots, y_d)\in \mathscr{T}$.
	We define the \textit{lexicographic order} on $\mathscr{T}$, $\mathcal{L}_{\mathscr{T}}$ to be a total order on $\mathscr{T}$,
	such that $x<_{\mathcal{L}_{\mathscr{T}}} y$ 
	iff for some $i\in \{1,\dots, d-1\}$ we have $x_1=y_1,\dots, x_i=y_i$ and $x_{i+1} <_{\mathscr{T}_{i+1}} y_{i+1}$.
	We abuse notation and treat $\mathscr{T}$ as its ground set, and define the toset $\mathscr{T}_{\mathcal{L}}=(\mathscr{T}, \mathcal{L}_{\mathscr{T}})$.
\end{dfn}

\begin{figure}
	\centering
	\begin{subfigure}[t]{0.16\textwidth}
		\includegraphics[width=\textwidth]{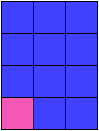}
	\end{subfigure}
	\hfill
	\begin{subfigure}[t]{0.16\textwidth}
		\includegraphics[width=\textwidth]{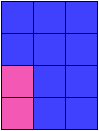}
	\end{subfigure}
	\hfill
	\begin{subfigure}[t]{0.16\textwidth}
		\includegraphics[width=\textwidth]{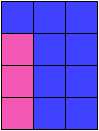}
	\end{subfigure}
	\hfill
	\begin{subfigure}[t]{0.16\textwidth}
		\includegraphics[width=\textwidth]{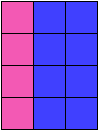}
	\end{subfigure}
	\hfill
	\begin{subfigure}[t]{0.16\textwidth}
		\includegraphics[width=\textwidth]{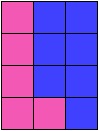}
	\end{subfigure}
	\hfill
	\begin{subfigure}[t]{0.16\textwidth}
		\includegraphics[width=\textwidth]{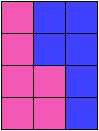}
	\end{subfigure}

	\vspace{0.5cm}
	\begin{subfigure}[t]{0.16\textwidth}
		\includegraphics[width=\textwidth]{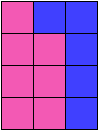}
	\end{subfigure}
	\hfill
	\begin{subfigure}[t]{0.16\textwidth}
		\includegraphics[width=\textwidth]{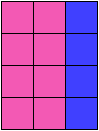}
	\end{subfigure}
	\hfill
	\begin{subfigure}[t]{0.16\textwidth}
		\includegraphics[width=\textwidth]{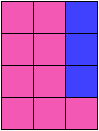}
	\end{subfigure}
	\hfill
	\begin{subfigure}[t]{0.16\textwidth}
		\includegraphics[width=\textwidth]{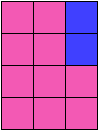}
	\end{subfigure}
	\hfill
	\begin{subfigure}[t]{0.16\textwidth}
		\includegraphics[width=\textwidth]{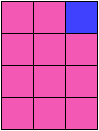}
	\end{subfigure}
	\hfill
	\begin{subfigure}[t]{0.16\textwidth}
		\includegraphics[width=\textwidth]{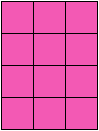}
	\end{subfigure}

	\caption{Lexicographic order in $\mathscr{M}_{[2]}(3,4)$.}\label{Lex_3x4}
\end{figure}

\begin{exm}
	A visualization of lexicographic order in $\mathscr{M}_{[2]}(3,4)$ can be seen in Figure \ref{Lex_3x4}.
	Notice how we fill up the space column by column.
\end{exm}

Note that any multiset lattice is a poset given by the standard partial order on the product of tosets (Proposition \ref{Decomposition_Of_Multiset_Latices}).
Lexicographic orders are total orders on multiset lattices.
We will often consider a partially ordered set equipped with another order (typically a total order) on its elements which is different from the partial order.
In the case of multiset lattices and lexicographic orders we have that the total order is an extension of the partial order, but we will not require this in general.
This observation leads us to the following definition.

For a positive integer $n$, an \textit{$n$-poset} is a tuple $\mathscr{P} = (S,\mathcal{O}_1,\dots, \mathcal{O}_n)$,
where $S$ is a set and for all $i\in [n]$ we have that $\mathcal{O}_i$ is a partial order on $S$.
For $d\in \N$ with $d\geq 1$, by $\mathfrak{S}_d$ we denote the \textit{set of all permutations on} $[d]$.
A subposet of an $n$-poset is just a restriction with respect to all partial orders.

\begin{dfn}[Domination Order]\label{Domination_Order_Definition}
	Suppose that $d\in \N$ with $d\geq 1$ and consider tosets $\mathscr{T}_1,\dots, \mathscr{T}_d$.
	Also, put $\mathscr{T}= \mathscr{T}_1\times \cdots \times \mathscr{T}_d$.
	For any $\pi \in \mathfrak{S}_d$ we define $\mathcal{D}_{\pi}$, the {\em domination order induced by $\pi$ on} $\mathscr{T}$,
	such that for any $x=(x_1,\dots, x_d),y=(y_1,\dots, y_d)\in \mathscr{T}$ we have
	\begin{align*}
		x \leq_{\mathcal{D}_\pi} y \text{ iff } (x_{\pi(1)},\dots, x_{\pi(d)}) \leq_{\mathscr{T}_{\mathcal{L}}} (y_{\pi(1)},\dots, y_{\pi(d)}).
	\end{align*}
	Then we define the toset $\mathscr{T}_{\pi} = (\mathscr{T}, \mathcal{D}_{\pi})$.
\end{dfn}

\begin{rem}
	The lexicographic order on a product of tosets is a domination order induced by the identity
	permutation, i.e. $\mathscr{T}_{\mathcal{L}} = \mathscr{T}_{\id_{[d]}}$.
\end{rem}

Domination orders are often employed.
The most common one of these orders, 
other than the lexicographic one, 
is the colexicographic order.

\begin{dfn}[Colexicographic Order]\label{Colexicographic_Order_Definition}
	Suppose that $d\in \N$ with $d\geq 1$ and consider tosets $\mathscr{T}_1,\dots, \mathscr{T}_d$
	Also, put $\mathscr{T}= \mathscr{T}_1\times \cdots \times \mathscr{T}_d$.
	Let $\pi \in \mathfrak{S}_d$ be the permutation such that $\pi(i) = d-i+1$ for all $i\in \{1,\dots, d\}$.
	The \textit{colexicographic order} $\mathcal{C}_{\mathscr{T}}$ on $\mathscr{T}$ is defined to be the domination order $\mathcal{D}_{\pi}$.
	We also write $\mathscr{T}_{\mathcal{C}}$ for $\mathscr{T}_\pi$.
	Another way to view the colexicographic order is to take any $x=(x_1,\dots, x_d),y=(y_1,\dots, y_d)\in \mathscr{T}$ and define $x<_{\mathcal{C}_{\mathscr{T}}} y$ 
	iff for some $i\in \{2,\dots, d+1\}$ we have $x_d=y_d,\dots, x_i=y_i$ and $x_{i-1} <_{\mathscr{T}_{i+1}} y_{i-1}$.
\end{dfn}

\begin{figure}
	\centering
	\begin{subfigure}[t]{0.19\textwidth}
		\includegraphics[width=\textwidth]{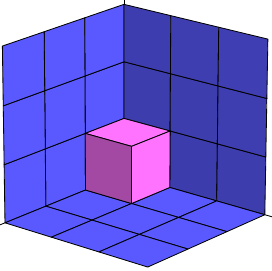}
	\end{subfigure}
	\hfill
	\begin{subfigure}[t]{0.19\textwidth}
		\includegraphics[width=\textwidth]{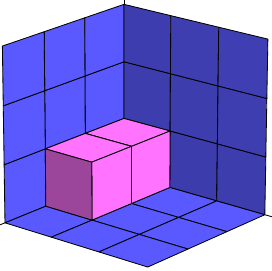}
	\end{subfigure}
	\hfill
	\begin{subfigure}[t]{0.19\textwidth}
		\includegraphics[width=\textwidth]{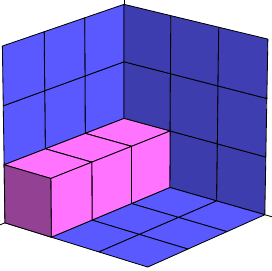}
	\end{subfigure}
	\hfill
	\begin{subfigure}[t]{0.19\textwidth}
		\includegraphics[width=\textwidth]{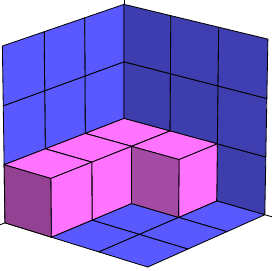}
	\end{subfigure}
	\hfill
	\begin{subfigure}[t]{0.19\textwidth}
		\includegraphics[width=\textwidth]{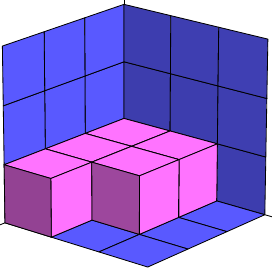}
	\end{subfigure}

	\vspace{0.1cm}
	
	\begin{subfigure}[t]{0.19\textwidth}
		\includegraphics[width=\textwidth]{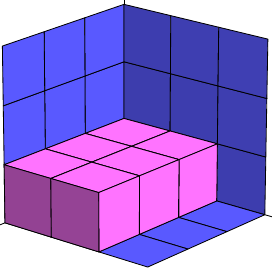}
	\end{subfigure}
	\hfill
	\begin{subfigure}[t]{0.19\textwidth}
		\includegraphics[width=\textwidth]{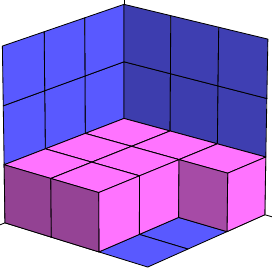}
	\end{subfigure}
	\hfill
	\begin{subfigure}[t]{0.19\textwidth}
		\includegraphics[width=\textwidth]{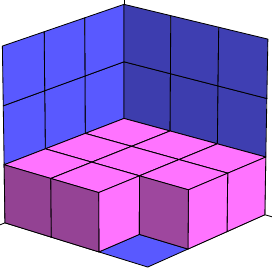}
	\end{subfigure}
	\hfill
	\begin{subfigure}[b]{0.19\textwidth}
		\centering
		\raisebox{0.5\hsize}{\includegraphics[width=0.5\textwidth]{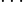}}
	\end{subfigure}
	\hfill
	\begin{subfigure}[t]{0.19\textwidth}
		\includegraphics[width=\textwidth]{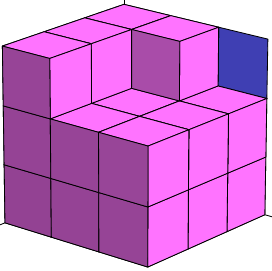}
	\end{subfigure}

	\vspace{0.1cm}
	
	\begin{subfigure}[t]{0.19\textwidth}
		\includegraphics[width=\textwidth]{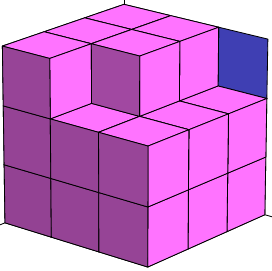}
	\end{subfigure}
	\hfill
	\begin{subfigure}[t]{0.19\textwidth}
		\includegraphics[width=\textwidth]{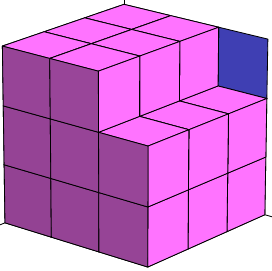}
	\end{subfigure}
	\hfill
	\begin{subfigure}[t]{0.19\textwidth}
		\includegraphics[width=\textwidth]{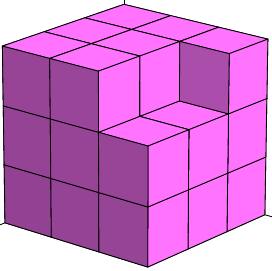}
	\end{subfigure}
	\hfill
	\begin{subfigure}[t]{0.19\textwidth}
		\includegraphics[width=\textwidth]{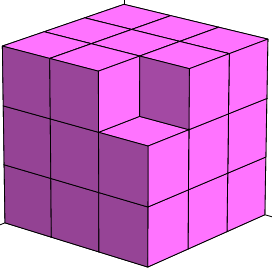}
	\end{subfigure}
	\hfill
	\begin{subfigure}[t]{0.19\textwidth}
		\includegraphics[width=\textwidth]{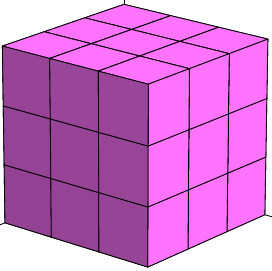}
	\end{subfigure}
	
	\caption{Colexicographic order in $\mathscr{M}_{[3]}(3,3,3)$.}
	\label{Colex_3x3x3}
\end{figure}

\begin{exm}
	A visualization of colexicographic order in $\mathscr{M}_{[3]}(3,3,3)$ can be seen in Figure \ref{Colex_3x3x3}.
\end{exm}

\begin{dfn}[Levels]\label{Level_Definition}
	Suppose that $\mathscr{P}$ is a ranked poset with rank function $r$.
	For $i\in \N$ the \textit{$i$-th level of} $\mathscr{P}$ is the set of all elements of rank $i$.
	This set will be denoted by $\Lvl_{i,r} = r^{-1}(\{i\})$.
	We will often just write $\Lvl_{i, \mathscr{P}}$ when the rank function is clear,
	and  $\Lvl_{i}$ when the poset is clear as well.
\end{dfn}

\begin{exm}
	The levels of $\mathscr{M}_{[d]}(3,3,3)$ can be seen in Figure \ref{Levels_3x3x3}.
\end{exm}

\begin{figure}
	\centering
	\begin{subfigure}[t]{0.23\textwidth}
		\includegraphics[width=\textwidth]{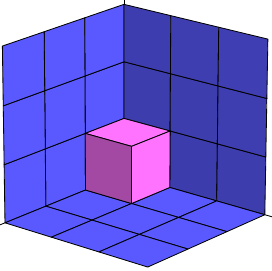}
		\subcaption{Rank $0$ elements.}
	\end{subfigure}
	\hfill
	\begin{subfigure}[t]{0.23\textwidth}
		\includegraphics[width=\textwidth]{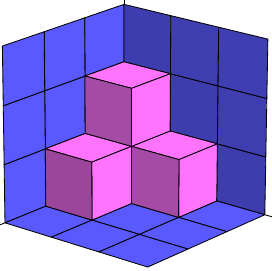}
		\subcaption{Rank $1$ elements.}
	\end{subfigure}
	\hfill
	\begin{subfigure}[t]{0.23\textwidth}
		\includegraphics[width=\textwidth]{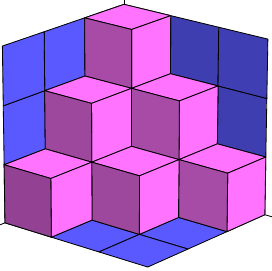}
		\subcaption{Rank $2$ elements.}
	\end{subfigure}
	\hfill
	\begin{subfigure}[t]{0.23\textwidth}
		\includegraphics[width=\textwidth]{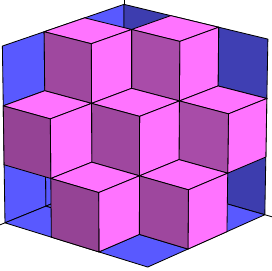}
		\subcaption{Rank $3$ elements.}
	\end{subfigure}

	\vspace{0.3cm}
	
	\begin{subfigure}[t]{0.23\textwidth}
		\includegraphics[width=\textwidth]{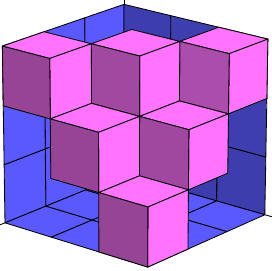}
		\subcaption{Rank $4$ elements.}
	\end{subfigure}
	\hfil
	\begin{subfigure}[t]{0.23\textwidth}
		\includegraphics[width=\textwidth]{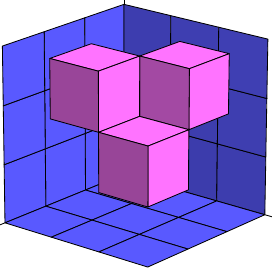}
		\subcaption{Rank $5$ elements.}
	\end{subfigure}
	\hfil
	\begin{subfigure}[t]{0.23\textwidth}
		\includegraphics[width=\textwidth]{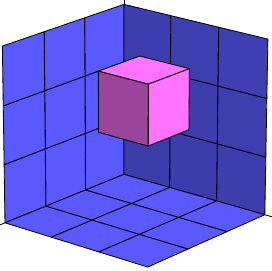}
		\subcaption{Rank $6$ elements.}
	\end{subfigure}
	
	\caption{The levels of $\mathscr{M}_{[3]}(3,3,3)$.}
	\label{Levels_3x3x3}
\end{figure}

\begin{dfn}[Shadows]\label{Shadows_Definition}
	Suppose that $\mathscr{P}$ is a poset and $a,b\in \mathscr{P}$.
	We say that $b$ is a \textit{lower shadow point} (resp. \textit{upper shadow point}) of $a$ iff $a$ covers $b$ (resp. $b$ covers $a$).
	By $\dSdw_{\mathscr{P}}(a)$ (resp. $\uSdw_{\mathscr{P}}(a)$) we denote the \textit{set of all lower shadow points (resp. upper shadow points) of} $a$ in $\mathscr{P}$.
	Similarly, for any $A\subseteq \mathscr{P}$ we define
	\begin{align*}
		\dSdw_{\mathscr{P}}(A) = \bigcup_{a\in A} \dSdw_{\mathscr{P}}(a) \text{ and }  \uSdw_{\mathscr{P}}(A) = \bigcup_{a\in A} \uSdw_{\mathscr{P}}(a).
	\end{align*}
	The subscript $\mathscr{P}$ will often be omitted.
\end{dfn}

\begin{dfn}[Dual of a Poset]
	For an $n$-poset $\mathscr{P} = (S,\mathcal{O}_1,\dots, \mathcal{O}_n)$ we define the \textit{dual} of $\mathscr{P}$ to be the $n$-poset $\mathscr{P}^\ast = (S, \mathcal{O}_1^\ast, \dots, \mathcal{O}_n^\ast)$,
	where for any $a,b\in S$ and any $i\in [n]$ we have $a\leq_{\mathcal{O}_i^\ast} b$ iff $b\leq_{\mathcal{O}_i} a$.
\end{dfn}

\begin{prp}[Properties of Duals]\label{Properties_of_Duals}
	For posets $\mathscr{P}, \mathscr{P}_1,\dots, \mathscr{P}_d$ and $A\subseteq \mathscr{P}$ we have:
	\begin{enumerate}
		\item $(\mathscr{P}_1\times \cdots \times \mathscr{P}_d)^\ast = \mathscr{P}_1^\ast \times \cdots \times \mathscr{P}_d^\ast$.
		\item $\dSdw_{\mathscr{P}}(A) = \uSdw_{\mathscr{P}^\ast}(A)$.
		\item $\uSdw_{\mathscr{P}}(A) = \dSdw_{\mathscr{P}^\ast}(A)$.
	\end{enumerate}
\end{prp}

\begin{exm}
	The dual of the lexicographic order restricted to $\Lvl_{2}$ of $\mathscr{M}_{[3]}(3,3,3)$ can be seen in Figure \ref{Level_2_Lex_Dual_3x3x3}.
\end{exm}

\begin{figure}
	\centering
	\begin{subfigure}[t]{0.23\textwidth}
		\includegraphics[width=\textwidth]{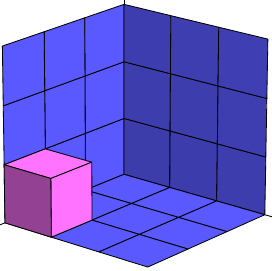}
	\end{subfigure}
	\hfil
	\begin{subfigure}[t]{0.23\textwidth}
		\includegraphics[width=\textwidth]{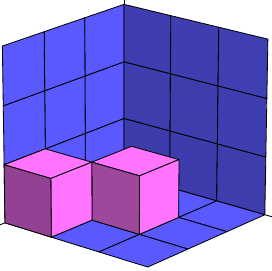}
	\end{subfigure}
	\hfil
	\begin{subfigure}[t]{0.23\textwidth}
		\includegraphics[width=\textwidth]{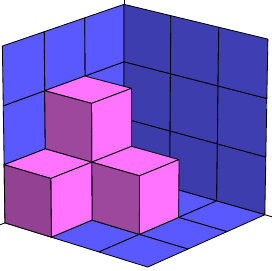}
	\end{subfigure}

	\vspace{0.3cm}
	
	\begin{subfigure}[t]{0.23\textwidth}
		\includegraphics[width=\textwidth]{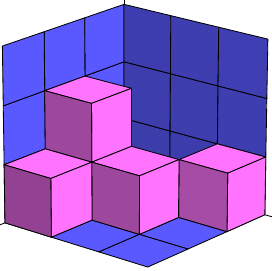}
	\end{subfigure}
	\hfil
	\begin{subfigure}[t]{0.23\textwidth}
		\includegraphics[width=\textwidth]{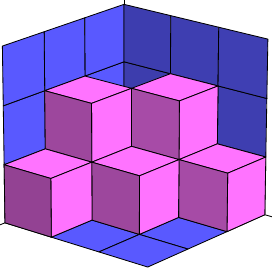}
	\end{subfigure}
	\hfil
	\begin{subfigure}[t]{0.23\textwidth}
		\includegraphics[width=\textwidth]{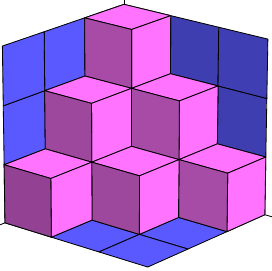}
	\end{subfigure}

	\caption{The dual of the lexicographic order restricted to $\Lvl_{2}$ of $\mathscr{M}_{[3]}(3,3,3)$.}
	\label{Level_2_Lex_Dual_3x3x3}
\end{figure}

\begin{dfn}\label{Ranked_Subposet}
	Suppose that $\mathscr{P}$ is ranked.
	Then we denote the \textit{ranked subposet up to rank} $n$ of $\mathscr{P}$ by $\mathscr{P}_{\leq n}$,
	and define it to be the restriction of $\mathscr{P}$ to
	\begin{align*}
		\bigcup_{i=0}^n \Lvl_{i}.
	\end{align*}
\end{dfn}

The following definition allows us to talk about different dimensional structures in a multiset lattice.

\begin{dfn}[Subproducts]\label{Subproducts_Definition}
	Suppose that $d\geq 1$ and we have tosets $\mathscr{T}_1,\dots, \mathscr{T}_d$, 
	and put $\mathscr{T}=\mathscr{T}_1\times \cdots \times \mathscr{T}_d$.
	For a set of coordinates $S=\{p_1 < \dots < p_k \}\subseteq [d]$ we define the \textit{subproduct of} $\mathscr{T}$ \textit{under} $S$ by
	\begin{align*}
		\mathscr{T}_S = \mathscr{T}_{p_1} \times \cdots \times \mathscr{T}_{p_k},
	\end{align*} 
	where $\mathscr{T}_{\emptyset} = \emptyset$.
	We say that $\mathscr{T}_S$ \textit{has dimension} $k$.
	Let $\overline{S} = [d]\setminus S$ and we can write $\overline{S}=\{q_1 < \cdots < q_{d-k}\}$.
	For $x=(x_{q_1},\dots, x_{q_{d-k}})\in \mathscr{T}_{\overline{S}}$ we define the \textit{subproduct at} $x$ \textit{under} $S$ \textit{of} $\mathscr{T}$ by
	\begin{align*}
		\mathscr{T}_S(x) = \{(y_1,\dots, y_d)\in \mathscr{T} \bigm | y_a=x_a \text{ for all } a\in \overline{S}\}.
	\end{align*}
\end{dfn}

\begin{exm}
	Some products can be seen in Figure \ref{subproductsExample}.
\end{exm}

\begin{figure}
	\centering
	\begin{subfigure}[t]{0.3\textwidth}
		\includegraphics[width=\textwidth]{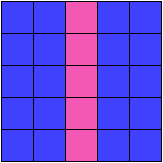}
		\caption{$(\mathscr{M}_{[2]}(5,5))_{\{2\}}(2)$.}
	\end{subfigure}
	\hfill
	\begin{subfigure}[t]{0.3\textwidth}
		\includegraphics[width=\textwidth]{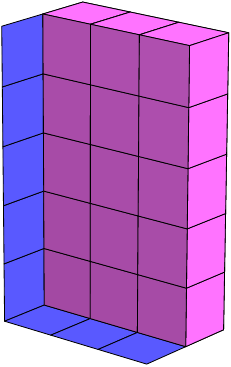}
		\caption{$(\mathscr{M}_{[3]}(2,3,5))_{\{2,3\}}(0)$.}
	\end{subfigure}
	\hfill
	\begin{subfigure}[t]{0.3\textwidth}
		\includegraphics[width=\textwidth]{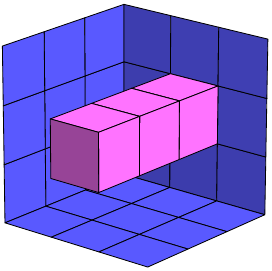}
		\caption{$(\mathscr{M}_{[3]}(3,3,3))_{\{1\}}((1,1))$.}
	\end{subfigure}
	
	\caption{Some subproducts of multiset lattices.}
	\label{subproductsExample}
\end{figure} \subsection{Macaulay Posets}\label{Macaulay_Posets}

Macaulay posets are posets for which similar results to the famous Kruskal--Katona Theorem hold \cite{Katona1966, kruskal_1963, cubeMacaulayFirst}.
They have been extensively studied in extremal combinatorics and entire chapters have been dedicated to them in textbooks \cite{engel_1997, HarperBook}.

\begin{dfn}[Macaulay Posets]\label{Macaulay_Posets_Definition}
	Suppose that we have a $2$-poset $\mathscr{T} = (S,\mathcal{O}_1, \mathcal{O}_2)$.
	Furthermore, suppose that $\mathcal{O}_2$ is a total order and $\mathscr{P} = (S,\mathcal{O}_1)$ is a ranked poset such that all the levels are finite.
	Then we can consider the levels of $\mathscr{P}$, and restrict the order $\mathcal{O}_2$ on $\Lvl_i$ for each $i\in [r(\mathscr{P})+1]_0$,
	thus creating a totally ordered set out of each level of $\mathscr{P}$.
	This in turn means that we can consider initial segments in each level of $\mathscr{P}$, 
	and write $\Lvl_i[q]$ to mean the initial segment of size $q$ of the $i$-th level of $\mathscr{P}$ under the order $\mathcal{O}_2$.
	For a set $A\subseteq \Lvl_{i}$ we define $\Lvl_{i}[A] = \Lvl_{i}[|A|]$.
	
	We say that $\mathscr{T}$ is \textit{Macaulay} if for any $i\in [r(\mathscr{P})]$ and any $A\subseteq \Lvl_i$ we have
	\begin{align*}
		\dSdw_{\mathscr{P}}(\Lvl_i[A]) \subseteq \Lvl_{i-1}[\dSdw_{\mathscr{P}}(A)].
	\end{align*}
	We will often say that $\mathcal{O}_2$ is a Macaulay order on $\mathscr{P}$, or that $\mathscr{P}$ forms a Macaulay poset with $\mathcal{O}_2$.
	Similarly, we will often write $(\mathscr{P}, \mathcal{O}_2)$ is Macaulay.
	Notation will sometimes be abused and we will treat $\mathscr{T}$ as the ranked poset $\mathscr{P}$.
	For example, we will write $r(\mathscr{T})$ to mean $r(\mathscr{P})$.
\end{dfn}

Macaulay posets are often discussed in an equivalent form.
The following lemma has been observed by many authors.

\begin{prp}[\cite{engel_1997}, Proposition 8.1.1]\label{Macaulay_Basic_Equivalence}
	Suppose that $\mathscr{P}$ is a ranked poset and $\mathcal{O}$ is a total order on $\mathscr{P}$.
	The $2$-poset $(\mathscr{P}, \mathcal{O})$ is Macaulay if and only if for each $i\in [r(P)]$ and any $A\subseteq \Lvl_{i}$ we have:
	\begin{enumerate}
		\item (Nestedness) Initial segments have the smallest shadows,
		\begin{align*}
			|\dSdw_{\mathscr{P}}(\Lvl_{i}[A])| \leq |\dSdw_{\mathscr{P}}(A)|.
		\end{align*}
		\item (Continuity) The shadow of an initial segment is an initial segment,
		\begin{align*}
			\dSdw_{\mathscr{P}}(\Lvl_{i}[A]) = \Lvl_{i-1}[\dSdw_{\mathscr{P}}(\Lvl_{i}[A])].
		\end{align*}
	\end{enumerate}
\end{prp}

We make another useful observation concerning Macaulay posets.
In algebraic terms, the following lemma says that Macaulay posets satisfy a direct limit property.

\begin{prp}\label{Macaulay_Direct_Limit}
	Suppose that for each $i\in \N$ we have a Macaulay poset $\mathscr{T}_i$.
	Furthermore, 
	for any $i\in \N$ suppose that $\mathscr{T}_i$ is a subposet of $\mathscr{T}_{i+1}$ with respect to both orders,
	the rank function of $\mathscr{T}_i$ is a restriction of the rank function of $\mathscr{T}_{i+1}$,
	and
	\begin{align*}
		\bigcup_{a=0}^i \mathscr{T}_a
	\end{align*}
	is finite.
	Then 
	\begin{align*}
		\mathscr{T} = \bigcup_{i=0}^\infty \mathscr{T}_i
	\end{align*}
	is Macaulay.
\end{prp}
\begin{proof}
	Definition \ref{Macaulay_Posets_Definition} is based on a property for each specific level.
	Since we have an ascending chain of subposets the claim follows immediately.
\end{proof}

The central idea of using Macaulay posets together with Hilbert functions rests on an observation made by Bezrukov concerning the duals of Macaulay posets.
A proof of the following lemma can be found in \cite{engel_1997} Proposition 8.1.2.

\begin{lem}[Bezrukov's Dual Lemma \cite{bezrukovDual} 1994]\label{Bezrukov_Dual_Lemma}
	Suppose that $\mathscr{P}$ is a ranked poset with $r(P)< \infty$, and that $\mathcal{O}$ is a total order on $\mathscr{P}$.
	We consider the $2$-poset $\mathscr{T} = (\mathscr{P}, \mathcal{O})$.
	Then $\mathscr{T}$ is Macaulay iff $\mathscr{T}^\ast$ is Macaulay. 
\end{lem}

In \cite{bezrukovDual}, general isoperimetric problems were considered.
The formulation of Lemma \ref{Bezrukov_Dual_Lemma} given above appears in \cite{engel_1997}.
If $r(\mathscr{P}) = \infty$ then $\mathscr{P}^\ast$ is not ranked, 
so $\mathscr{P}^\ast$ can't be Macaulay because it is not even ranked.
However, we can still extract a lot of information about $\mathscr{P}^\ast$ if we know that $\mathscr{P}$ is Macaulay.

\begin{lem}\label{Macaulay_Equivalence_Strong}
	Suppose that $\mathscr{T} = (P, \mathcal{O}', \mathcal{O})$, 
	where $\mathscr{P} = (P, \mathcal{O}')$ is a ranked poset,
	$\mathcal{O}$ is a total order on $P$,
	and every level is finite.
	The following are equivalent:
	\begin{enumerate}
		\item $\mathscr{T}$ is Macaulay.
		\item For every $n\in \N$, $\mathscr{T}_{\leq n}$ is Macaulay.
		\item For any $i\in [n]_0$ and any $A\subseteq \Lvl_{i, \mathscr{P}}$ we have
		\begin{align*}
			\uSdw_{\mathscr{P}}(\Lvl_{i,\mathscr{P}}^\ast[A]) &\subseteq \Lvl^\ast_{i+1, \mathscr{P}}[\uSdw_{\mathscr{P}}(A)],
		\end{align*}
		where $\Lvl_{i,\mathscr{P}}^\ast$ means to take the $i$-th level of $P$, 
		impose the total order induced by $\mathcal{O}$, 
		and then take the dual.
		\item For any $i\in [n]_0$ and any $A\subseteq \Lvl_{i, \mathscr{P}}$ we have
		\begin{align*}
			|\uSdw_{\mathscr{P}}(\Lvl^\ast_{i,\mathscr{P}}[A])| &\leq |\uSdw_{\mathscr{P}}(A)|,\\
			\uSdw_{\mathscr{P}}(\Lvl^\ast_{i, \mathscr{P}}[A]) &= \Lvl_{i+1, \mathscr{P}}^\ast[\uSdw_{\mathscr{P}}(\Lvl^\ast_{i, \mathscr{P}}[A])].
		\end{align*}
	\end{enumerate}
\end{lem}
\begin{proof}
	We have that (1) implies (2) by Definition \ref{Macaulay_Posets_Definition}.
	Proposition \ref{Macaulay_Direct_Limit} gives us that (2) implies (1).
	For all $n\in \N$, $r(\mathscr{T}_{\leq n})< \infty$, so we have that $\mathscr{T}^\ast_{\leq n}$ is Macaulay by Bezrukov's Dual Lemma \ref{Bezrukov_Dual_Lemma}.
	Thus, (2) is equivalent to (3) by Proposition \ref{Properties_of_Duals}
	Finally, (3) is equivalent to (4) by Lemma \ref{Macaulay_Basic_Equivalence} and Proposition \ref{Properties_of_Duals}.
\end{proof}

The oldest result on Macaulay posets is due to Macaulay.

\begin{thm}[Macaulay (the dual problem was considered) \cite{MacaulayF.S.1927SPoE} 1927]\label{Macaulay_1927}
	If $\mathscr{M} = \mathscr{M}_{[d]}(\infty,\dots, \infty)$ then $(\mathscr{M}, \mathcal{L}_{\mathscr{M}})$ is Macaulay.
\end{thm}

A famous result concerning Macaulay posets is the Kruskal--Katona Theorem.
This result first appeared in print in 1959 \cite{cubeMacaulayFirst} by Schützenberger, but the proof is incomplete.
Later it was proved by Kruskal in 1963 \cite{kruskal_1963}, and independently discovered by Katona in 1966 \cite{Katona1966}.
It is a special case of Harper's vertex isoperimetric inequality from 1966 \cite{Harper1966OptimalNA}.

\begin{thm}\label{KruskalKatona}
	If $\mathscr{M} = \mathscr{M}_{[d]}(2,\dots , 2)$ then $(\mathscr{M}, \mathcal{L}_{\mathscr{M}})$ is Macaulay. 
\end{thm}

Clements and Lindström completely settled the Macaulay problem for finite multiset lattices.

\begin{thm}[Clements--Lindström \cite{ClementsG.F.1969Agoa} 1969]\label{Clements_Lindstrom_Theorem}
	If $\mathscr{M} = \mathscr{M}_{[d]}(\ell_1,\dots ,\ell_d)$ with $\ell_1 \leq \ell_2 \leq \cdots \leq \ell_d < \infty$ then $(\mathscr{M}, \mathcal{L}_{\mathscr{M}})$ is Macaulay.
\end{thm}

We must remark that the above three classical results of Macaulay posets were originally not stated like above.
Macaulay worked with monomials in $K[x_1,\dots, x_d]$.
Clements and Lindström worked with upper shadows and the dual of the lexicographic order.
There are many proofs of the Kruskal--Katona Theorem using different techniques.
Many authors did not use the lexicographic order, but some domination order, which in a lot of cases was the colexicographic order. 

Notice that Macaulay's Theorem follows from the Clements--Lindström Theorem by using Proposition \ref{Macaulay_Direct_Limit}.
Similarly, we get the following corollary.

\begin{cor}\label{Macaulay_Full_Multiset_Theorem}
	If $\mathscr{M}$ is any multiset lattice then $\mathscr{M}$ forms a Macaulay poset with some domination order. 
\end{cor}
 \subsection{The Poset of Monomials}\label{The_Poset_Of_Monomials}

We assume that all rings are commutative and have $1$.
$K$ will always denote a field.
A graded ring is a ring $R$ together with a decomposition $R = \bigoplus_{i=0} R_i$ (as an abelian group) such that $R_iR_j \subseteq R_{i+j}$ for all $i,j \in N$. 
We say that an element $x\in R$ is \textit{homogeneous} if there exists $i\in \N$ such that $x\in R_i$.
Then for any nonzero $f\in R$ we can uniquely write $f=f_1+\cdots + f_n$, 
where each nonzero $f_j$ belongs to some $R_{i_j}$.
The ring elements $f_1,\dots, f_n$ are called the \textit{homogeneous components} of $f$.
For any nonzero $x\in R_i$ we define the \textit{degree} of $x$ in $R$ to be $\deg_R(x) = i$.

\begin{dfn}[Homogeneous Ideal]\label{Homogeneous_Ideal_Definition}
	An ideal $I$ in a graded ring $R$ is said to be \textit{homogeneous} if and only if $I$ is generated by homogeneous elements.
	If $R$ is a graded ring and $H$ is a homogeneous ideal of $R$ then $S=R/H$ is graded with $\Lvl_{i,S} = \frac{\Lvl_{i,R}+H}{H}$.
	In particular, for any homogeneous $f\in R$ such that $f+H$ is nonzero, we have $\deg_S(f+H) = \deg_R(f)$.
	We say that this is the induced grading on $R/H$ from $R$.
\end{dfn}

\begin{prp}[Properties of Homogeneous Ideals]\label{Properties_of_Homogeneous_Ideals}
	Suppose that $R=\bigoplus_{i=0}^\infty R_i$ is a graded ring and $I$ is an ideal of $R$.
	The following are equivalent:
	\begin{enumerate}
		\item $I$ is homogeneous.
		\item $I=\bigoplus_{i=0}^\infty (R_i \cap I)$.
		\item Let $f\in R$. One has, $f\in I$ iff each of the homogeneous components of $f$ are in $I$.
	\end{enumerate}
\end{prp}

\begin{dfn}[Homogeneous Components/Levels in Graded Rings]\label{Levels_in_Graded_Rings_Definition}
	If $R=\bigoplus_{i=0}^\infty R_i$ is a graded ring, then we define the \textit{$i$-th level/homogeneous component} of $R$ by $\Lvl_{i,R} = R_i$.
\end{dfn}

The \textit{standard grading} on $R=K[x_1,\dots, x_d]$ is given by making $\Lvl_{i,R}$ be all homogeneous polynomials of degree $i$.
Then for a homogeneous ideal $H\subseteq R$ we define the \textit{standard grading} on $R/H$ to be induced by the standard grading on $R$.
For the rest of this section let $R=K[x_1,\dots, x_d]$ and $S=R/H$, where $H$ is a homogeneous ideal of $R$.
Furthermore, we assume that $H\neq R$, which means $1\not\in H$.

A \textit{monomial} of $S$ is a nonzero element that has the form $x_1^{p_1}\cdots x_d^{p_d} + H$, where $p_1,\dots, p_d\in \N$.
An ideal $H\subseteq S$ is called \textit{monomial} if it is generated by monomials. 
Every monomial ideal in $S$ is a homogeneous ideal.

\begin{dfn}[Set of Monomials and Monomial Division]\label{Set_of_Monomials_and_Monomial_Division_Definition}
	The \textit{set of all monomials} of $S$ will be denoted by $M_S$.
	For $f,g\in M_S$ we say $f|g$ iff there exists $a\in M_S$ such that $g=af$.
	We say that $f$ \textit{divides} $g$.
\end{dfn}

\begin{prp}\label{Monomial_Partial_Order}
	Monomial division is a partial order on $M_S$.
\end{prp}
\begin{proof}
	If $q\in M_S$ then $q|q$, since $q=(1+H)q$, and note that $1+H\in M_S$ because $H\neq R$.
	Thus, monomial division is reflexive.
	
	Next, we handle antisymmetry.
	So, suppose that $p,q\in M_S$ with $p|q$ and $q|p$.
	Then there exist $a,b\in M_S$ such that $q=ap$ and $p=bq$.
	Hence, $\deg(q)=\deg(a)+\deg(p)$ and $\deg(p)=\deg(b)+\deg(q)$.
	Thus, $\deg(q) = \deg(a)+\deg(b)+\deg(q)$,
	which gives $0=\deg(a)+\deg(b)$.
	As $\deg(a),\deg(b)\geq 0$ we must have that $\deg(a)=\deg(b)=0$.
	However, there is only one monomial of degree $0$, which forces $a=b=1+H$.
	Therefore, $p=q$ as desired.
	
	We are left to handle transitivity.
	Suppose that $p,q,m\in M_S$ such that $p|q$ and $q|m$.
	Then there are $a,b\in M_S$ such that $q=ap$ and $m=bq$.
	Thus, $m=(ba)p$.
	Well, $ba\neq 0$, since otherwise $m=0$.
	Hence, $ba\in M_S$, and we have that $p|m$.
	So, monomial division is transitive.
	
	Therefore, monomial division is a partial order on $M_S$.
\end{proof}

\begin{dfn}[Poset of Monomials]\label{Poset_of_Monomials_Definition}
	Considering Proposition \ref{Monomial_Partial_Order}, monomial division will be called the \textit{monomial partial order} on $S$ from now on.
	The \textit{poset of monomials} of $S$ is defined to be the set of all monomials together with the monomial partial order,
	and we denote by $\mathscr{M}_S$.
	Furthermore, we define the \textit{rank} of an element of $\mathscr{M}_S$ to be its degree,
	and this makes $\mathscr{M}_S$ into a ranked poset.
\end{dfn}

\begin{lem}\label{Monomial_Partial_Order_Representatives}
	Suppose that $S=K[x_1,\dots, x_d]/H$, 
	where $H$ is a homogeneous ideal,
	and let $m_1,m_2\in \mathscr{M}_S$.
	We have that $m_1|m_2$ iff there are $p_1,\dots, p_d,q_1,\dots, q_d\in \N$ such that
	\begin{enumerate}
		\item $m_1 = x_1^{p_1}\cdots x_d^{p_d} + H$.
		\item $m_2 = x_1^{q_1}\cdots x_d^{q_d} + H$.
		\item $p_1\leq q_1, \dots ,p_d\leq q_d$.
	\end{enumerate}
\end{lem}
\begin{proof}
	The $\Longleftarrow$ direction follows from the definition of monomial division in $K[x_1,\dots, x_d]$.
	So, suppose that $m_1|m_2$ and we prove the $\implies$ direction.
	Then there exists $a\in \mathscr{M}_S$ such that $m_2=am_1$.
	Thus, 
	\begin{align*}
		m_1 &= x_1^{p_1}\cdots x_d^{p_d} + H,\\
		a &= x_1^{t_1}\cdots x_d^{t_d} + H,
	\end{align*}
	for some $p_1,\dots, p_d,t_1,\dots, t_d\in \N$.
	Then $m_2 = x_1^{p_1+t_1}\cdots x_d^{p_d+t_d}+H$,
	and the claim holds by setting $q_1 = p_1+t_1,\dots, q_d = p_d+t_d$.
\end{proof}

The poset of monomials of a ring has many nice properties.
One crucial property is that its shadow functions correspond to multiplying or dividing by the variables.
That is, the partial order is compatible with the multiplication of the ring.
A rigorous statement of this observation is given for the upper shadows in the next lemma, and we leave the downwards shadow version as an exercise.

\begin{lem}[Upper Shadow in The Poset of Monomials]\label{Upper_Shadow_in_The_Poset_of_Monomials}
	If $A\subseteq \mathscr{M}_S$ then
	\begin{align*}
		\uSdw (A) = \bigcup_{j=1}^d \{m(x_j+H) \bigm | m\in A \text{ and } m(x_j+H) \neq 0\}.
	\end{align*}
\end{lem}
\begin{proof}
	The backwards inclusion follows from the definition of $\mathscr{M}_S$.
	So, we prove the forward inclusion.
	Let $p\in \uSdw(A)$.
	Then there exists $m\in A$ such that $m|p$ and $\deg(m) +1 = \deg(p)$.
	Hence, $p=am$ for some $a\in \mathscr{M}_S$.
	However, $\deg(m) +1 = \deg(p)$ forces $\deg(a)=1$.
	Thus, $a=x_j+H$ for some $j\in[d]$.
	Therefore the claim holds.
\end{proof}

We are going to make an even stronger statement about the upper shadow function in Lemma \ref{Monomial_Space_iff_Ideal},
but before that we need a few more definitions.

\begin{dfn}[Graded Vector Spaces]\label{Graded_Vector_Spaces_Definition}
	A $K$-subspace $V$ of $S$ is said to be \textit{graded} if for every $i\in \N$ there is a $K$-subspace $V_i$ of $\Lvl_{i,S}$ such that
	\begin{align*}
		V = \bigoplus_{i=0}^\infty V_i.
	\end{align*}
	We call $V_i$ the $i$-th \textit{graded component/level} of $V$ and write $\Lvl_{i,V} = V_i$.
\end{dfn}

\begin{prp}
	If $I$ is a homogeneous ideal of $S$ then $I$ is a graded vector space.
\end{prp}
\begin{proof}
	Follows from Proposition \ref{Properties_of_Homogeneous_Ideals}.
\end{proof}

\begin{dfn}[Monomial Spaces]\label{Monomial_Spaces_Definition}
	A graded vector space is called a \textit{monomial space} if it has a basis consisting of monomials.
\end{dfn}

\begin{lem}\label{Monomial_Space_iff_Ideal}
	A monomial space $V\subseteq S$ is an ideal iff for all $i\in \N$ we have $\uSdw(\Lvl_{i,V}\cap \mathscr{M}_S)\subseteq \Lvl_{i+1,V}\cap \mathscr{M}_S$.
\end{lem}
\begin{proof}
	The forward direction follows by Lemma \ref{Upper_Shadow_in_The_Poset_of_Monomials}.
	So, suppose that for all $i\in \N$ we have $\uSdw(\Lvl_{i,V}\cap \mathscr{M}_S)\subseteq \Lvl_{i+1,V}\cap \mathscr{M}_S$.
	Then Lemma \ref{Upper_Shadow_in_The_Poset_of_Monomials} implies that for every $m\in \Lvl_{i,V}\cap \mathscr{M}_S$ and for every $j\in [d]$ we have that $m(x_j+H)\in \Lvl_{i+1,V}\cap \mathscr{M}_S$,
	since $V$ being a group implies $0\in V$.
	Therefore, with a few induction arguments we get that $V$ is an ideal.
\end{proof}

\begin{dfn}\label{Downsets_and_Upsets_Definition}
	For an arbitrary poset $A\subseteq \mathscr{P}$ we say that $A$ is a \textit{downset} iff whenever $x\in A$ and $y\leq x$ then $y\in A$.
	Similarly, we say that $A$ is an \textit{upset} iff whenever $x\in A$ and $x\leq y$ then $y\in A$.
	Downsets are sometimes called order ideals.
\end{dfn}

Using the results so far, 
one can make some useful observations about the poset of monomials of a ring.
First, notice that $\mathscr{M}_R\cap H$ is an upset.
Second, if we have a monomial $m\in H$ then $m+H=0\in S$.
Thus, Lemma \ref{Upper_Shadow_in_The_Poset_of_Monomials} and Lemma \ref{Monomial_Space_iff_Ideal} say that if we have a monomial $m\in H$,
then the upset defined by $m$ in $\mathscr{M}_R$ is deleted in $\mathscr{M}_S$.
This deletion operation can be seen in Figure \ref{Poset_of_Monomials_Example_a}.
Using this we can conclude the poset isomorphism in Example \ref{Grid_Ring_Poset}.

\begin{dfn}[Infinite Variable Powers]\label{Infinite_Variable_Powers_Definition}
	For $i\in [d]$ we define $x_i^\infty = 0\in R$.
\end{dfn}

\begin{exm}\label{Grid_Ring_Poset}
	If $H=(x_1^{p_1},\dots, x_d^{p_d})$ for some $p_1,\dots, p_d\in \N \cup \{\infty\}$ then $\mathscr{M}_S \iso \mathscr{M}_{[d]}(p_1,\dots, p_d)$.
\end{exm}

Another useful observation from these ideas is that domination orders on $\mathscr{M}_S$ are well-defined.
Something even stronger is true.

\begin{prp}[Monomial Quotients Produce Subposets]\label{Monomial_Quotients_Produce_Subposets}
	If $H$ is a monomial ideal then $\mathscr{M}_S$ is isomorphic to a subposet of $\mathscr{M}_R$.
	In particular, if $\mathcal{O}$ is a total order on $\mathscr{M}_R$, 
	then there exists a well-defined order on $\mathscr{M}_S$ induced from the isomorphism and $\mathcal{O}$.
\end{prp}
\begin{proof}
	Since $H$ is a monomial ideal, the function $\mathscr{M}_S \rightarrow \mathscr{M}_R$ that sends $x_1^{p_1}\cdots x_d^{p_d}+H$ to $x_1^{p_1}\cdots x_d^{p_d}$ is injective,
	whence we get that $\mathscr{M}_S$ is isomorphic to a subposet of $\mathscr{M}_R$.
\end{proof}

Another useful connection to make is to understand what happens when we have binomials $m_1-m_2\in H$, where $m_1,m_2\in \mathscr{M}_R$.
Well, in this case we have $(m_1-m_2)+H=0\in S$, whence $m_1+H=m_2+H$.
That is, we consider the two monomials $m_1$ and $m_2$ in $\mathscr{M}_R$ and identify them as the same element in the poset $\mathscr{M}_S$.
This can be seen in Figure \ref{Poset_of_Monomials_Example_b}

\begin{figure}
	\centering
	\begin{subfigure}[t]{0.40\textwidth}
		\includegraphics[width=\textwidth]{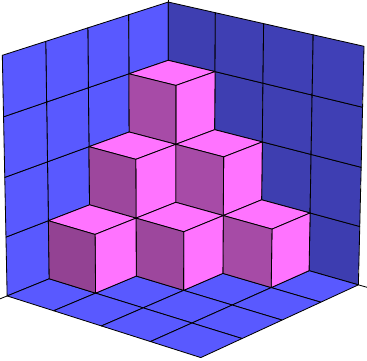}
		\caption{Cube diagram of $\mathscr{M}_S$ when $H=(\Lvl_3)$ and $d=3$.}\label{Poset_of_Monomials_Example_a}
	\end{subfigure}
	\hfill
	\begin{subfigure}[t]{0.55\textwidth}
		\includegraphics[width=\textwidth]{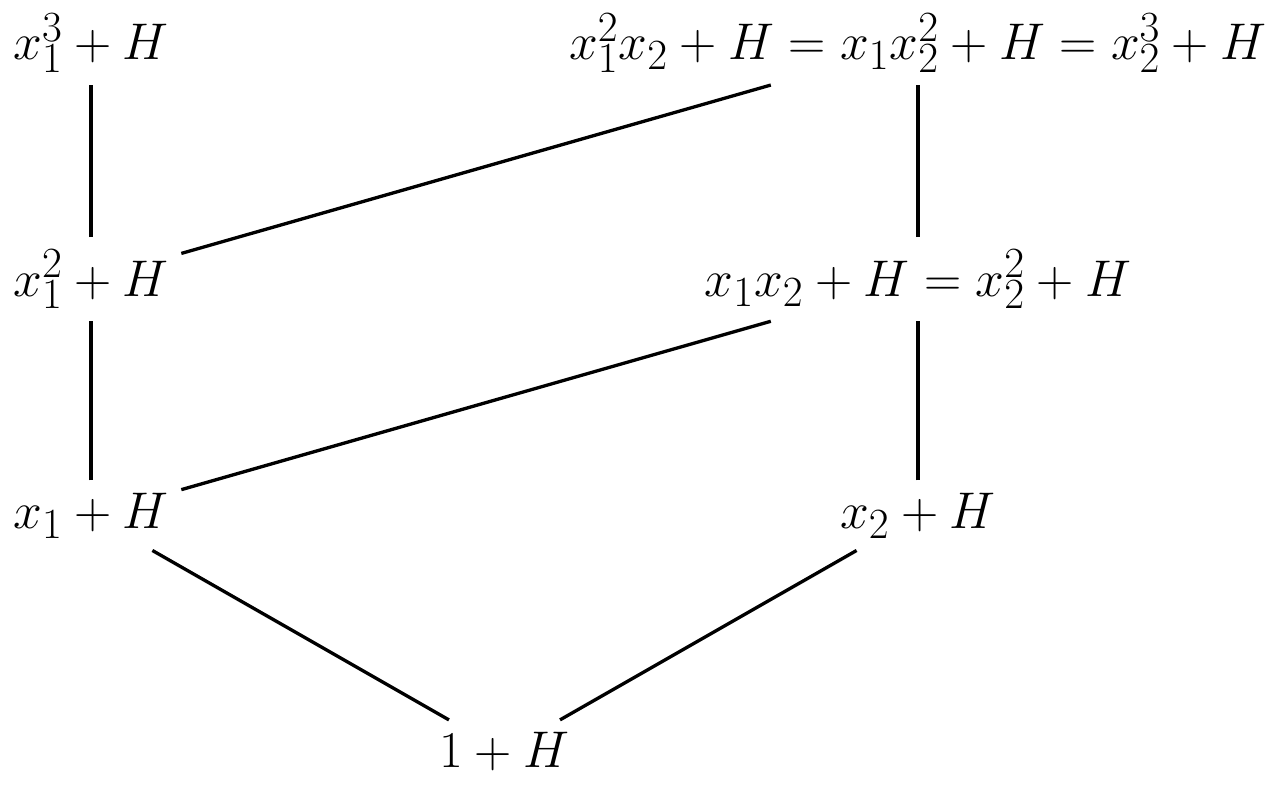}
		\caption{Hasse Graph of $\mathscr{M}_S$ when $H=(\Lvl_4) + (x_1x_2-x_2^2)$ and $d=2$.}\label{Poset_of_Monomials_Example_b}
	\end{subfigure}
	\hfill
	
	\caption{Posets of monomials for different $H$.}
	\label{Poset_of_Monomials_Example}
\end{figure}

\begin{dfn}[Posets Representable by Rings]\label{Posets Representable by Rings}
	We say that a poset $\mathscr{P}$ is \textit{representable by a ring} iff there exists a ring $S$ such that $\mathscr{P} \iso \mathscr{M}_S$.
\end{dfn}

By Example \ref{Grid_Ring_Poset} we have that any multiset lattice is representable by a ring. \subsection{Initial Monomial Vector Spaces and Ideals}\label{Initial_Monomial_Vector_Spaces_and_Ideals}

In this section, $R=K[x_1,\dots, x_d]$ and $S=R/H$, where $H$ is a homogeneous ideal of $R$ and $K$ is a field.
Furthermore, we assume that $H\neq R$, which means $1\not\in H$.

\begin{dfn}[Hilbert Functions]\label{Hilbert_Functions_Definition}
	For a graded subspace $V$ of $S$ we define the \textit{Hilbert function} of $V$ to be $\Hilb_V:\N\rightarrow \N$ such that for all $i\in \N$ we have
	\begin{align*}
		\Hilb_V(i) = \dim_K \Lvl_{i,V}.
	\end{align*}
\end{dfn}

If $V$ is a vector space over a field $K$ and $A\subseteq V$,
then by $\Span_K(A)$ we denote the span of the set $A$.

\begin{rem}
	In commutative algebra, 
	the initial monomial of a ring element $f$ is defined to be the largest monomial with non-zero coefficient, 
	and w.r.t. a monomial order (not an arbitrary total order).
	In this paper we use the smallest monomial and arbitrary total orders.
\end{rem}

\begin{dfn}[Initial Monomial Sets, Vector Spaces and Ideals]\label{Initial_Monomial_Sets_Vector_Spaces_and_Ideals_Definition}
	Suppose that $\mathcal{O}$ is a total order on $\mathscr{M}_{S}$.
	Note that $\Lvl_{i,R} = \Span_K(\Lvl_{i, \mathscr{M}_R})$,
	whence $\Lvl_{i, S} = \Span_K(\Lvl_{i, \mathscr{M}_S})$.
	Thus, each $\Lvl_{i,S}$ has a basis consisting of monomials.
	Let $B_i=\{m_{1,i},\dots, m_{n_i,i}\}$ be a basis of monomials for $\Lvl_{i,S}$.
	Then for any nonzero $f\in \Lvl_{i,S}$ we can uniquely write $f=a_{f,1}m_{1,i} + \cdots + a_{f,n_i}m_{n_i,i}$,
	such that $a_{f,1},\dots, a_{f,n_i}\in K$, $m_{1,i},\dots, m_{n_i,i}\in \mathscr{M}_S$ and $m_{1,i} \leq_{\mathcal{O}} \cdots \leq_{\mathcal{O}} m_{n_i,i}$.
	The \textit{initial monomial} of $f$ under $\mathcal{O}$ and $B_i$ is defined to be the smallest $m_{j,i}$ (with respect to $\mathcal{O}$) such that $a_{f,j} \neq 0$, 
	and we denote it by $\im_{\mathcal{O},B_i}(f)$.
	Next, let $B=(B_i)_{i=0}^\infty$, and we call $B$ a \textit{leveled monomial basis} of $S$.
	For a homogeneous ideal $I\subseteq S$ we define:
	\begin{enumerate}
		\item The \textit{initial monomial set of degree} $i\in \N$ under $\mathcal{O}$ and $B_i$,
		\begin{align*}
			\IMS_{\mathcal{O},B_i}(i,I) = \{\im_{\mathcal{O},B_i}(f)\bigm | f\in \Lvl_{i,I}\}.
		\end{align*}
		\item The \textit{initial monomial set} under $\mathcal{O}$ and $B$,
		\begin{align*}
			\IMS_{\mathcal{O},B}(I) = \bigcup_{i=0}^\infty \IMS_{\mathcal{O},B_i}(i,I).
		\end{align*}
		\item The \textit{initial monomial vector space} $\mathcal{O}$ and $B$,
		\begin{align*}
			\IMV_{\mathcal{O},B}(I) = \bigoplus_{i=0}^\infty \Span_K(\IMS_{\mathcal{O},B_i}(i,I)).
		\end{align*}
		\item The \textit{initial monomial ideal} under $\mathcal{O}$ and $B$, which is the ideal generated by $\IMS_{\mathcal{O},B}(I)$, 
		\begin{align*}
			\IMI_{\mathcal{O},B}(I) = \left( \IMS_{\mathcal{O},B}(I) \right).
		\end{align*}
	\end{enumerate}
	We will often omit indices as the total order and leveled monomial basis will be fixed. 
\end{dfn}

\begin{lem}[General Hilbert Function Equalities and Inequalities]\label{General_Hilbert_Function_Equalities_and_Inequalities}
	Suppose that $\mathcal{O}$ is a total order on $S$ and $B$ is a leveled monomial basis for $S$.
	If $I\subseteq S$ is a homogeneous ideal then $\Hilb_I(i) = \Hilb_{\IMV(I)}(i) \leq \Hilb_{\IMI(I)}(i)$ for all $i\in \N$.
\end{lem}
\begin{proof}
	The inequality follows from the Standard Grading on $S$ and the relation $\IMV(I) \subseteq \IMI(I)$.
	We deal with the equality next.
	
	For each $i\in \N$ there exists a basis of homogeneous polynomials $C_i=\{f_{1,i},\dots, f_{k_i,i}\}$ for $\Lvl_{i,I}$.
	Then for each $i\in \N$ and $f\in C_i$ we can write $f = a_{f,1}m_{1}+\cdots + a_{f,n_i}m_{n_i}$,
	where $B_i = \{m_{1},\dots, m_{n_i}\}$ and $m_{1}<_{\mathcal{O}} \dots <_{\mathcal{O}} m_{n_i}$.
	
	Then we form the matrix $M_i$, such that its rows are the coefficients of the homogeneous polynomials in $C_i$ with respect to $B_i$ and $\mathcal{O}$.
	For $1\leq r \leq k_i$ and $1\leq c \leq n_i$ we define
	\begin{align*}
		M_i = (a_{f_{r,i},c})_{r,c}
	\end{align*}
	Let $M_i'$ be the row reduced echelon form of $M_i$.
	Thus, the row spaces of $M_i$ and $M_i'$ are identical,
	and the nonzero rows of $M_i'$ are a basis for the row space of $M_i$,
	which is isomorphic to $\Span_K(C_i)$ as $K$-vector spaces.
	Since $C_i$ is a basis for $\Span_K(C_i)$, 
	the number of non-zero rows of $M_i'$ must equal $|C_i| = k_i$, 
	hence $M_i'$ does not have any zero rows.
	Let $f_{j,i}'$ be the homogeneous polynomial corresponding to the $j$-th row of $M_{i}'$ and put $C_i'=\{f_{1,i}',\dots, f_{k_i,i}'\}$.
	Observe that $C_i'$ is a basis for $Span_K(C_i)$ = $Lvl_{i,I}$.
	Then $C_i'\subseteq \Lvl_{i, I}$ and we define $B_i'=\{\im(f_{1,i}'),\dots, \im(f_{k_i,i}')\} \subseteq \Lvl_{i, \IMV(I)}$.
	
	First, note that $|B_i'|=k_i$,
	since the leading entries of the rows in $M_i'$ are the only nonzero entries in their own columns.
	Next, $B_i'$ is linearly independent,
	since $B_i'\subseteq B_i$ and $B_i$ is linearly independent.
	We show that $\Span_K(B_i') = \Lvl_{i, \IMV(I)}$, whence $B_i'$ is a basis for $\Lvl_{i, \IMV(I)}$.
	The containment $\subseteq$ follows right away, since $B_i'\subseteq \Lvl_{i, \IMV(I)}$.
	Thus, we just need to show $\IMS(i,I)\subseteq \Span_K(B_i')$.
	So, let $m\in \IMS(i,I)$.
	Then there exists $f\in \Lvl_{i,I}$ such that $m=\im(f)$.
	Thus, we can write $f=b_1f_{1,i}'+\cdots b_{k_i}f_{k_i,i}'$.
	Let $j\in [k_i]$ be the smallest integer such that $b_j\neq 0$.
	Hence, $m=\im(f_{j,i}')$ because $B_i$ is linearly independent and the leading entries of the rows in $M_i'$ are the only nonzero entries in their own columns.
	Therefore, $\Span_K(B_i') = \Lvl_{i, \IMV(I)}$ and $B_i'$ is a basis for $\Lvl_{i, \IMV(I)}$.
	
	Finally,
	\begin{align*}
		\Hilb_I(i) = |C_i| = |C_i'|=|B_i'|=\Hilb_{\IMV(I)}(i).
	\end{align*}
\end{proof}

Example \ref{IMIStict} shows that the inequality in Lemma \ref{General_Hilbert_Function_Equalities_and_Inequalities} can be strict.

\begin{exm}\label{IMIStict}
	Consider $S=K[x,y]/(0)$.
	We define a total order $\mathcal{O}$ on $\mathscr{M}_S$ that totally orders by degree first, 
	then the monomials of degree $1$ are ordered based on lexicographic order $\mathcal{L}$,
	and finally the monomials of degree other than $1$ are ordered based on the colexicographic order $\mathcal{C}$. 
	More formally, $m_1 \leq_{\mathcal{O}} m_2$ if and only if
	\begin{enumerate}
		\item $\deg(m_1) < \deg(m_2)$,
		\item or $\deg(m_1) = \deg(m_2)$ and
		\begin{enumerate}
			\item if $\deg(m_1) = \deg(m_2) = 1$ then $m_1 \leq_{\mathcal{L}} m_2$.
			\item if $\deg(m_1) = \deg(m_2) \neq 1$ then $m_1 \leq_{\mathcal{C}} m_2$.
		\end{enumerate}
	\end{enumerate}
	Now consider the homogeneous ideal $I=(x+y)$.
	Then $x^2+xy\in I$ and $xy+y^2\in I$ and let $B_I = \{x^2+xy,xy+y^2\}$.
	Suppose that $f\in I$ is homogeneous and $\deg(f)=2$.
	Then $f=g(x+y)$ for some homogeneous $g\in S$ with $\deg(g)=1$.
	Thus, $g=ax+by$ for some $a,b\in K$.
	Hence, $f\in \Span_K(B_I)$.
	Therefore, $\Hilb_{I}(2)=2$.
	
	Next, we consider $\IMI(I)$ with the order $\mathcal{O}$, and note that there is a unique leveled monomial basis of $S$.
	Well, $y\in \IMI(I)$ because $x+y\in I$, whence $y^2,xy\in \IMI(I)$.
	Also, $x^2+xy\in I$, whence $x^2\in \IMI(I)$.
	So, $\Hilb_{\IMI(I)}(2) = 3$.
	
	Therefore, $\Hilb_I(2) = \Hilb_{\IMV(I)}(2) < \Hilb_{\IMI(I)}(2)$.
	This situation occurs because $\mathcal{O}$ is not monomial order (Definition \ref{Monomial_Order_Definition}).
\end{exm}

We would like to have $\Hilb_{I}=\Hilb_{\IMI(I)}$.
The following lemma tells us when this happens.

\begin{lem}[General Hilbert Function Equalities]\label{General_Hilbert_Function_Equalities}
	Suppose that $\mathcal{O}$ is a total order on $S$ and $B$ is a leveled monomial basis for $S$.
	Furthermore, let $I\subseteq S$ be a homogeneous ideal of $S$.
	$\IMV(I)$ is an ideal iff $\IMV(I) = \IMI(I)$ iff $\Hilb_I = \Hilb_{\IMV(I)} = \Hilb_{\IMI(I)}$.
\end{lem}
\begin{proof}
	First, we handle the forward directions.
	If $\IMV(I)$ is an ideal then $\IMV(I) = \IMI(I)$, since $\IMV(I)\subseteq \IMI(I)$ and $\IMV(I)$ contains the generating set of $\IMI(I)$.
	If $\IMV(I) = \IMI(I)$ then $\Hilb_I = \Hilb_{\IMV(I)} = \Hilb_{\IMI(I)}$ from Lemma \ref{General_Hilbert_Function_Equalities_and_Inequalities}.
	
	The backwards directions holds because $\dim_K \Lvl_{i,\IMV(I)} = \dim_K \Lvl_{i, \IMI(I)}$  and $\Lvl_{i,\IMV(I)} \subseteq \Lvl_{i, \IMI(I)}$ implies $\Lvl_{i,\IMV(I)} = \Lvl_{i, \IMI(I)}$.
	Thus, $\IMV(I)=\IMI(I)$.
	Furthermore, $\IMV(I)=\IMI(I)$ implies that $\IMV(I)$ is an ideal.
\end{proof}

When assumed, the following two properties will force $\IMV(I)$ to be an ideal.

\begin{dfn}[Monomial Order]\label{Monomial_Order_Definition}
	A total order $\mathcal{O}$ on $\mathscr{M}_S$ is called a \textit{monomial order} if whenever we have $m_1,m_2\in \mathscr{M}_S$ such that $m_1< m_2$,
	and we have $m\in \mathscr{M}_S$ such that $mm_1,mm_2\in \mathscr{M}_S$,
	then $mm_1 < mm_2$.
\end{dfn}

Monomial orders are very common.
For example, any domination order on $\mathscr{M}_R$ is a monomial order.
Considering Proposition \ref{Monomial_Quotients_Produce_Subposets} we see that this implies that there is always a monomial order on $\mathscr{M}_S$ when $H$ is a monomial ideal.
The main property of monomial orders that we are interested in is their existence.
In all of our results, the specific monomial order does not matter and is independent of the other properties.
When a monomial order computing the Hilbert function of a homogeneous ideal can be reduced to the case of computing the Hilbert function of a monomial ideal,
by using Theorem \ref{Reducing_Homogeneous_Ideals_to_Monomial_Ideals}.

\begin{dfn}[Level Linear Independence]\label{Level_Linear_Independence_Definition}
	We say that $S$ is \textit{level linearly independent} if for all $i\in \N$ we have that $\Lvl_{i,\mathscr{M}_S}$ is linearly independent.
	We will often say that $\mathscr{M}_{S}$ is level linearly independent and we just mean that $S$ is level linearly independent.
\end{dfn}

Definition \ref{Level_Linear_Independence_Definition} is important to the proof of the Macaulay Correspondence Theorem \ref{Macaulay_Correspondence_Theorem}.
As the following result shows, it can also be used to preserve the Hilbert function when constructing initial monomial ideals.

\begin{thm}[Reducing Homogeneous Ideals to Monomial Ideals]\label{Reducing_Homogeneous_Ideals_to_Monomial_Ideals}
	Suppose that $\mathcal{O}$ is a monomial order on $\mathscr{M}_S$ and suppose $S$ is level linearly independent.
	If $I$ is a homogeneous ideal of $S$ then $\Hilb_I = \Hilb_{\IMV(I)} = \Hilb_{\IMI(I)}$.
\end{thm}
\begin{proof}
	Note that there is a unique leveled monomial basis of $S$, since $S$ is level linearly independent.
	Because of Lemma \ref{General_Hilbert_Function_Equalities} we only need to show that $\IMV(I)$ is an ideal.
	Thus, for any $g\in S$ and any $m\in IMV(I)$ we need to show that $gm\in \IMV(I)$.
	Well, $g$ is a finite sum of homogeneous elements, 
	each of which is a finite sum of scalar multiples of monomials in $S$.
	Also, $m$ is a finite sum of homogeneous elements, 
	each of which is a finite sum of scalar multiples of initial monomials, 
	and each of the initial monomials is $\im(f)$ for some homogeneous element $f\in I$.
	Hence, $gm$ is a sum, such that each term in the sum is a product is a scalar multiple between a monomial and an initial monomial of some homogeneous element of $I$.
	We prove the claim when $g\in \mathscr{M}_S$ and $m\in \IMS(I)$, whence the general case will follow immediately afterwards.
	
	So, we assume that $g\in \mathscr{M}_S$ and $m\in \IMS(I)$, 
	and we want show that $gm\in \IMV(I)$.
	Also, assume that $gm\neq 0$, since $0\in \IMV(I)$.
	We can write $m=\im(f)$ for some homogeneous $f\in I$.
	Then we have $f = a_1m_1+\cdots + a_nm_n$,
	where $a_1,\dots, a_n\in K\setminus \{0\}$, 
	the terms $m_1,\dots, m_n$ are monomials of $\deg(f)$ in $S$,
	and we have $m_1< \cdots <  m_n$.
	Of course, $m=\im(f) = m_1$.
	Since $I$ is an ideal we have $gf\in I$,
	and since $g$ and $f$ are homogeneous we have that $gf$ is homogeneous.
	
	Well,
	\begin{align*}
		gf = a_1gm_1+\cdots +a_ngm_n.
	\end{align*}
	Note that $a_1gm_1 \neq 0$,
	otherwise $gm=gm_1=0$ ($a_1\neq 0$) a contradiction.
	Thus, $gf \neq 0$, since $S$ is level linearly independent.
	So, $\im(gf)\in \IMV(I)$.
	We show that $\im(gf)=g\im(f)=gm=gm_1$.
	First, for each $j\in [n]$ such that $gm_j\neq 0$ we have that $gm_j$ is a monomial in the basis of degree $\deg(gf)$ monomials,
	since $S$ is level linearly independent.
	That is, the expression (after the appropriate ordering) given above for $gf$ is used in the construction of $\IMS(I)$, after eliminating $0$ terms.
	Furthermore, for any $j_1,j_2\in [n]$ such that $gm_{j_1},gm_{j_2}\neq 0$ we have $gm_{j_1} < gm_{j_2}$,
	since $\mathcal{O}$ is a monomial order.
	All of this together gives $\im(gf)=g\im(f)=gm=gm_1$.
	
	Therefore, the proof is complete.
\end{proof}

\begin{cor}\label{Reducing_Homogeneous_Ideals_to_Monomial_Ideals_Corollary}
	Suppose that $H$ is a monomial ideal.
	If $I$ is a homogeneous ideal of $S$ then $\Hilb_I = \Hilb_{\IMV(I)} = \Hilb_{\IMI(I)}$.
\end{cor}
\begin{proof}
	The lexicographic order on $\mathscr{M}_R$ is a monomial order, whence Proposition \ref{Monomial_Quotients_Produce_Subposets} gives us a monomial order on $\mathscr{M}_S$.
	Since $H$ is a monomial ideal we have that $S$ is level linearly independent.
	Therefore, the claim holds by Theorem \ref{Reducing_Homogeneous_Ideals_to_Monomial_Ideals}.
\end{proof}

 \subsection{Macaulay Rings}\label{Macaulay_Rings}

Unless otherwise stated $S=K[x_1,\dots, x_d]/H$, where $K$ is a field and $H$ is a homogeneous ideal of $R=K[x_1,\dots, x_d]$.
Furthermore, we assume that $H\neq R$, which means $1\not\in H$.

\begin{dfn}[Initial Segment Spaces]\label{initialSegmentSpace}
	Suppose that $V$ is a graded vector subspace of $S$ and let $\mathcal{O}$ be a total order on $\mathscr{M}_S$.
	We define the \textit{initial $\mathcal{O}$-segment space} of $V$ by (note that we defined $\Lvl_{i, \mathscr{M}_S}^\ast$ in Lemma \ref{Macaulay_Equivalence_Strong})
	\begin{align*}
		\mathcal{O}^\ast[V] = \bigoplus_{i=0}^\infty \Span_K(\Lvl_{i, \mathscr{M}_S}^\ast[\dim_K \Lvl_{i,V}]).
	\end{align*}
\end{dfn}

The Macaulay Correspondence Theorem \ref{Macaulay_Correspondence_Theorem} gives us conditions for when $\mathcal{O}^\ast[V]$ is an ideal.

\begin{rem}
	For a homogeneous ideal $I$, 
	$\mathcal{L}_{\mathscr{M}_S}^\ast[I]$ is known as a lexicographic ideal or lex ideal in the literature,
	when it holds that $\mathcal{L}_{\mathscr{M}_S}^\ast[I]$ is an ideal.
\end{rem}

\begin{exm}\label{hilbNotEqualExample}
	In general we have $\Hilb_{V}(i) \geq \Hilb_{\mathcal{O}^\ast[V]}(i)$ for all $i\in \N$.
	We could have $\Hilb_{V}(i) > \Hilb_{\mathcal{O}^\ast[V]}(i)$ for some $i$, when $\Lvl_{i, \mathscr{M}_S}^\ast[\dim_K \Lvl_{i,V}]$ is not linearly independent.
	Take $R=K[x,y,z]$ with $d=3$, $H=(x^2+xy-xz)$ and $S=R/H$.
	Then we consider the ideal $I=(z^2+H, yz+H, y^2 + H)$.
	Notice that $\{z^2+H, yz+H, y^2 + H\}$ is linearly independent.
	Hence, $\Hilb_I(2)=3$.
	However, $\{x^2+H, xy+H, xz+H\}$ is not linearly independent, since $(x^2+H)+(xy+H)=xz+H$.
	Also, $\mathcal{L}^\ast[I]\cap \mathscr{M}_S = \{x^2+H, xy+H, xz+H\}$.
	Thus, $\Hilb_{\mathcal{L}^\ast[I]}(2)=2$.
	Therefore, $\Hilb_I(2) > \Hilb_{\mathcal{L}^\ast[I]}(2)$ and $\Hilb_I\neq \Hilb_{\mathcal{L}^\ast[I]}$.
\end{exm}

Example \ref{hilbNotEqualExample} reinforces the importance of Definition \ref{Level_Linear_Independence_Definition} for preserving Hilbert functions.

\begin{lem}\label{levelLIimpliesHilb}
	If $S$ is level linearly independent and $\mathcal{O}$ is a total order on $\mathscr{M}_S$ then for any graded vector subspace $V$ of $S$ we have
	\begin{align*}
		\Hilb_{V} = \Hilb_{\mathcal{O}^\ast[V]}
	\end{align*}
\end{lem}
\begin{proof}
	Follows from Definition \ref{initialSegmentSpace} and Definition \ref{Level_Linear_Independence_Definition}.
\end{proof}

\begin{lem}\label{levelLImonoLemma}
	If $S$ is level linearly independent and $\mathcal{O}$ is a total order on $\mathscr{M}_S$, 
	then for any graded vector subspace $V$ of $S$ and all $i\in \N$ we have
	\begin{align*}
		\Lvl_{i, \mathscr{M}_S}^\ast[\dim_K \Lvl_{i,V}] = \Lvl_{i, \mathcal{O}^\ast[V]}\cap \mathscr{M}_S.
	\end{align*}
\end{lem}
\begin{proof}
	One has, $\subseteq$ right away from Definition \ref{initialSegmentSpace}.
	
	So, let $m\in \Lvl_{i, \mathcal{O}^\ast(V)}\cap \mathscr{M}_S$.
	Then 
	\begin{align*}
		m\in \Span_K (\Lvl_{i, \mathscr{M}_S}^\ast[\dim_K \Lvl_{i,V}])\cap \mathscr{M}_S.
	\end{align*}
	Thus, $m$ is a linear combination of elements of $\Lvl_{i, \mathscr{M}_S}^\ast[\dim_K \Lvl_{i,V}]$.
	However, $S$ is level linearly independent. 
	Hence, every linear combination of $\Lvl_{i, \mathscr{M}_S}^\ast[\dim_K \Lvl_{i,V}]$ that involves at least two terms is not a monomial.
	Thus, we must have $m\in \Lvl_{i, \mathscr{M}_S}^\ast[\dim_K \Lvl_{i,V}]$.
	
	Therefore, the claim holds.
\end{proof}

Mermin and Peeva in \cite{MerminJeffrey2006Li, MerminJeffrey2007Hfal} asked about conditions on $H$ (for monomial ideals) such that for any homogeneous ideal $I$ of $S$,
there exists a lexicographic ideal $L$ of $S$ such that $\Hilb_I = \Hilb_L$.
They called such rings $S$ Macaulay-Lex, and $H$ a Macaulay-Lex ideal of $R$.
We are interested in generalizing this problem in two directions.
First, we want to allow an arbitrary order, not just the lexicographic one.
Second, we want to take quotients by homogeneous ideals, not just monomial ideals.

\begin{dfn}[Macaulay Rings and Ideals]\label{Macaulay_Rings_and_Ideals_Definition}
	Suppose that $\mathcal{O}$ is a total order on $\mathscr{M}_S$.
	We say that $(S, \mathcal{O})$ is \textit{Macaulay} if for every homogeneous ideal $I$ of $S$ we have that $\mathcal{O}^\ast[I]$ is an ideal,
	and $\Hilb_I = \Hilb_{\mathcal{O}^\ast[I]}$.
	If $(S, \mathcal{O})$ is Macaulay then we say $H$ is \textit{Macaulay}.
	We will sometimes say that $S$ is \textit{Macaulay} without specifying the total order,
	or say that $S$ is \textit{Macaulay} with $\mathcal{O}$.
\end{dfn}

\begin{prb}[The General Mermin--Peeva Problem]\label{The_General_Mermin_Peeva_Problem}
	Find classes of Macaulay rings.
\end{prb} \subsection{The Macaulay Correspondence Theorem}\label{Macaulay_implies_Hilbert}

Unless otherwise stated $S=K[x_1,\dots, x_d]/H$, where $K$ is a field and $H$ is a homogeneous ideal of $R=K[x_1,\dots, x_d]$.

\begin{lem}\label{Macaulay_Ring_to_Poset}
	Suppose that $S$ is level linearly independent and $\mathcal{O}$ is a total order on $\mathscr{M}_S$.
	If $(S,\mathcal{O})$ is Macaulay then $(\mathscr{M}_S, \mathcal{O})$ is Macaulay.
\end{lem}
\begin{proof}
	First, from the definitions we have that $\mathscr{M}_S$ is ranked and all its levels are finite.
	Let $i\in \N$ and $A\subseteq \Lvl_{i, \mathscr{M}_S}$.
	Then let $I$ be the ideal generated by $A$.
	Thus, $\mathcal{O}^\ast[I]$ is an ideal, since we assumed that $(S,\mathcal{O})$ is Macaulay.
	One has,
	\begin{align*}
		\uSdw(\Lvl_{i,\mathscr{M}_S}^\ast[A])
		&= \uSdw(\Lvl_{i, \mathscr{M}_S}^\ast[\dim_K \Lvl_{i,I}]) \tag{By level linear independence}\\ 
		&=\uSdw (\Lvl_{i, \mathcal{O}^\ast[I]}\cap \mathscr{M}_S) \tag{By Lemma \ref{levelLImonoLemma}}\\
		&\subseteq \Lvl_{i+1, \mathcal{O}^\ast[I]} \cap \mathscr{M}_S \tag{By Lemma \ref{Monomial_Space_iff_Ideal}}\\
		&= \Lvl_{i+1, \mathscr{M}_S}^\ast[\dim_K \Lvl_{i+1,I}] \tag{By Lemma \ref{levelLImonoLemma}}\\
		&=\Lvl^\ast_{i+1, \mathscr{M}_S}[\uSdw(A)] \tag{By level linear independence and definition of $I$}.
	\end{align*}
\end{proof}

\begin{lem}\label{Macaulay_Poset_to_Ring_Only_Monomial_Ideals}
	Suppose that $S$ is level linearly independent and $\mathcal{O}$ is a total order on $\mathscr{M}_S$.
	If $M\subseteq S$ is a monomial ideal and $(\mathscr{M}_S, \mathcal{O})$ is Macaulay then $\mathcal{O}^\ast[M]$ is an ideal.
\end{lem}
\begin{proof}
	For all $i\in \N$ one has
	\begin{align*}
		\uSdw(\Lvl_{i,\mathcal{O}^\ast[M]}\cap \mathscr{M}_S) 
		&= \uSdw(\Lvl_{i, \mathscr{M}_S}^\ast[\dim_K \Lvl_{i,M}])\tag{By Lemma \ref{levelLImonoLemma}}\\
		&= \uSdw(\Lvl_{i, \mathscr{M}_S}^\ast[\Lvl_{i, M}\cap \mathscr{M}_S])\tag{By level linear independence}\\ 
		&\subseteq \Lvl_{i+1, \mathscr{M}_S}^\ast[\uSdw(\Lvl_{i, M}\cap \mathscr{M}_S)])\tag{By Lemma \ref{Macaulay_Equivalence_Strong}}\\
		&\subseteq \Lvl_{i+1, \mathscr{M}_S}^\ast[\Lvl_{i+1, M}\cap \mathscr{M}_S])\tag{Since $\uSdw(\Lvl_{i, M}\cap \mathscr{M}_S) \subseteq \Lvl_{i+1, M}\cap \mathscr{M}_S$}\\
		&=\Lvl_{i+1,\mathscr{M}_S}^\ast [\dim_K \Lvl_{i+1,M}]\tag{By level linear independence}\\
&= \Lvl_{i+1,\mathcal{O}^\ast[M]}\cap \mathscr{M}_S. \tag{By Lemma \ref{levelLImonoLemma}}
	\end{align*}
\end{proof}

We can now put all the developments in the previous sections into proving the Macaulay Correspondence Theorem \ref{Macaulay_Correspondence_Theorem}.

\begin{thm}[Macaulay Correspondence Theorem]\label{Macaulay_Correspondence_Theorem}
	Suppose that $S$ is level linearly independent, $\mathcal{O}$ is a total order on $\mathscr{M}_S$, and there exists a monomial order on $\mathscr{M}_S$.
	Then $(S,\mathcal{O})$ is Macaulay if and only if $(\mathscr{M}_S, \mathcal{O})$ is Macaulay.
\end{thm}
\begin{proof}
	The forward direction follows by Lemma \ref{Macaulay_Ring_to_Poset}.
	The backwards direction is slightly more complicated.
	Suppose that $(\mathscr{M}_S,\mathcal{O})$ is Macaulay and $I\subseteq S$ is a homogeneous ideal.
	Well, $\Hilb_{I} = \Hilb_{\mathcal{O}^\ast[I]}$ by Lemma \ref{levelLIimpliesHilb}.
	So, we just need to show that $\mathcal{O}^\ast[I]$ is an ideal.
	
	Let $B$ be the unique leveled monomial basis of $S$ and $\mathcal{O}'$ be a monomial order on $\mathscr{M}_S$.
	Then Theorem \ref{Reducing_Homogeneous_Ideals_to_Monomial_Ideals} gives us that $M=\IMI_{\mathcal{O}',B}(I)$ and $I$ have the same Hilbert function.
	Note that $M$ is a monomial ideal.
	Thus, Lemma \ref{Macaulay_Poset_to_Ring_Only_Monomial_Ideals} implies that $\mathcal{O}^\ast[M]$ is an ideal.
	However, $\mathcal{O}^\ast[M]=\mathcal{O}^\ast[I]$, 
	since $\Hilb_M=\Hilb_I$ and the sets $\mathcal{O}^\ast[M]$ and $\mathcal{O}^\ast[I]$ only depend on $\Hilb_M$ and $\Hilb_I$ respectively.
	Hence, $\mathcal{O}^\ast[I]$ is an ideal.
	Therefore, $(S,\mathcal{O})$ is Macaulay. 
\end{proof}

A couple remarks are in order for the proof of the Macaulay Correspondence Theorem \ref{Macaulay_Correspondence_Theorem}.
Level linear independence is crucial for many parts of several proofs that contributed to the final result.
The existence of a monomial order on $\mathscr{M}_S$ is only needed to reduce a homogeneous ideal to a monomial ideal in the backwards direction.
In particular, if $H$ was a monomial ideal, we don't need to worry about this condition because of Corollary \ref{Reducing_Homogeneous_Ideals_to_Monomial_Ideals_Corollary}.
Furthermore, if $H$ is a monomial ideal then it forces $S$ to be level linearly independent.
Therefore, we have the following corollary.

\begin{cor}[Macaulay Correspondence Theorem for Monomial Quotients]\label{Macaulay_Correspondence_Theorem_for_Monomial_Quotients}
	Suppose that $H$ is a monomial ideal.
	Then $(S,\mathcal{O})$ is Macaulay iff $(\mathscr{M}_S, \mathcal{O})$ is Macaulay.
\end{cor}

Although the author discovered Theorem \ref{Macaulay_Correspondence_Theorem} independently,
it should mention that a less general version of it has been discovered before.
Shakin in \cite{ShakinDA2001SgoM} proves the case when $H$ is a monomial ideal and $\mathcal{O} = \mathcal{L}$.
Notice that we can combine Theorem \ref{Reducing_Homogeneous_Ideals_to_Monomial_Ideals} with Lemma \ref{Macaulay_Equivalence_Strong} to obtain even more equivalences.
A special case (again $H$ is a monomial ideal and $\mathcal{O}=\mathcal{L}$) of this was observed by Chong in \cite{ChongKaiFongErnest2015Hfoc}.
Chong used Shakin's result and Proposition 8.1.1 (Bezrukov's Dual Lemma \ref{Bezrukov_Dual_Lemma}) in Engel's book \cite{engel_1997}.
Engel mentions Macaulay's Theorem \ref{Macaulay_1927} in \cite{engel_1997},
but does not discuss Hilbert functions at all.

With the general version, Theorem \ref{Macaulay_Correspondence_Theorem}, 
we can deduce many results about Macaulay rings by using the existing theory on Macaulay posets.
To end this section, we give the following corollary as simple example of what is to come in later sections.
 
 \begin{cor}\label{Clements_Lindström_Macaulay_Theorem}
 	Let $p_1,\dots ,p_d\in (\N\cup \infty)\setminus \{0\}$.
 	If $H=(x_1^{p_1},\dots, x_d^{p_d})$ then $S$ is Macaulay with some domination order. 
 \end{cor}
\begin{proof}
	Follows from Corollary \ref{Macaulay_Full_Multiset_Theorem} and Corollary \ref{Macaulay_Correspondence_Theorem_for_Monomial_Quotients}.
\end{proof}

\begin{dfn}
	The rings in Corollary \ref{Clements_Lindström_Macaulay_Theorem} are called \textit{Clements--Lindström rings},
	and in the special case when $p_1=\cdots =p_d=2$ we call them \textit{Kruskal-Katona rings}.
\end{dfn} \subsection{Products of Posets and Rings}\label{Products_of_Posets_and_Rings}

Results on Macaulay posets are often stated in terms of Cartesian products of other Macaulay posets.
For example, any multiset lattice is just the product of totally ordered sets (Decomposition of Multiset Lattices \ref{Decomposition_Of_Multiset_Latices}), 
each of which is Macaulay with the lexicographic order.
Since we will often use results about Macaulay posets to deduce results about Hilbert functions, it is useful to have an analog of products of posets for rings.
In this section we translate the Cartesian product operation on posets to the tensor product on rings.
A very useful result concerning the product of two vector spaces is the following lemma, and is Theorem 14.5 in \cite{RomanBook}.

\begin{lem}\label{Tensor_Linear_Independence}
	Suppose that $U$ and $V$ are vector spaces over $K$.
	If $u_1,\dots, u_n$ are linearly independent vectors in $U$ and $v_1,\dots, v_n$ are arbitrary vectors in $V$,
	then in $U \otimes_K V$ we have
	\begin{align*}
		\sum_{i=1}^n u_i\otimes v_i = 0 \implies v_i=0 \text{ for all } i\in [n].
	\end{align*}
	In particular, $u\otimes v = 0$ iff $u=0$ or $v=0$.
\end{lem}

Lemma \ref{Tensor_Linear_Independence} is very useful when proving the poset properties of the next lemma.
Parts 1-4 of Lemma \ref{tensorsBasic} are standard facts, 
and included for completeness.

\begin{lem}\label{tensorsBasic}
	Let $R_x = K[x_1,\dots, x_{d_x}]$, $R_y = K[y_1,\dots, y_{d_y}]$ and consider the ring that has both classes of variables $R_{xy}=K[x_1,\dots, x_{d_x},y_1,\dots, y_{d_y}]$.
	Suppose $H_x$ and $H_y$ be homogeneous ideals of $R_x$ and $R_y$ respectively.
	Consider $S_x = R_x/H_x$, $S_y = R_y/H_y$, and  $S_{xy}=R_{xy}/(H_x'+H_y')$,
	where $H_x'$ is the image of $H_x$ under the natural inclusion map $R_x\xhookrightarrow{} R_{xy}$,
	and $H_y'$ is the image of $H_y$ under the natural inclusion map $R_y\xhookrightarrow{} R_{xy}$.
	Then:
	\begin{enumerate}
		\item $S_x\otimes_K S_y$ is a ring, 
		where the group structure comes form the definition of tensor product, 
		and multiplication is given by (naturally extending for sums of basic tensors)
		\begin{align*}
			((r_x+H_x)\otimes (r_y+H_y))((r_x'+H_x)\otimes (r_y'+H_y)) = (r_xr_x' + H_x)\otimes (r_yr_y' + H_y).
		\end{align*}
		\item There exists a ring homomorphism $\sigma: S_x\otimes_K S_y\rightarrow S_{xy}$
		such that
		\begin{align*}
			\sigma((r_x+H_x) \otimes (r_y+H_y)) = r_xr_y + (H_x'+H_y').
		\end{align*}
		\item There exists a ring homomorphism $\tau: S_{xy}\rightarrow S_x\otimes_K S_y$ such that
		\begin{align*}
			\tau (x_1^{p_1}\cdots x_{d_x}^{p_{d_x}}y_1^{q_1}\cdots y_{d_y}^{q_{d_y}} + (H_x'+H_y')) = (x_1^{p_1}\cdots x_{d_x}^{p_{d_x}}+H_x)\otimes (y_1^{q_1}\cdots y_{d_y}^{q_{d_y}} + H_y).
		\end{align*}
		\item We have that $\sigma$ and $\tau$ are inverses of each other and hence we have an isomorphism of rings
		\begin{align*}
			S_x\otimes_K S_y \iso S_{xy}.
		\end{align*}
		\item If every pair of distinct monomials of the same degree in $S_x$ is linearly independent then we have an isomorphism of posets
		\begin{align*}
			\mathscr{M}_{S_x} \times \mathscr{M}_{S_y} \iso \mathscr{M}_{S_{xy}}.
		\end{align*}
		\item If $\mathscr{M}_{S_x}$ and $\mathscr{M}_{S_y}$ are level linearly independent then $\mathscr{M}_{S_{xy}}$ is level linearly independent.
		\item If every pair of monomials of the same degree in $S_x$ is linearly independent, 
		and there exist monomial orders on $\mathscr{M}_{S_x}$ and $\mathscr{M}_{S_y}$, 
		then there exists a monomial order on $\mathscr{M}_{S_{xy}}$.
	\end{enumerate}
\end{lem}
\begin{proof}

The first four claims are standard in commutative algebra.
	The fifth claim is straightforward from the the fourth claim.
	
	We prove the sixth claim.
	Suppose that $\mathscr{M}_{S_x}$ and $\mathscr{M}_{S_x}$ are level linearly independent.
	Take distinct monomials $m_{x_1}m_{y_1}+(H_x'+H_y'),\dots, m_{x_n}m_{y_n}+(H_x'+H_y') \in \mathscr{M}_{S_{xy}}$ and $a_1,\dots, a_n\in K$ such that 
	\begin{align*}
		a_1m_{x_1}m_{y_1}+(H_x'+H_y')+ \cdots + a_nm_{x_n}m_{y_n}+(H_x'+H_y') = 0.
	\end{align*}
	By using $\tau$ we get 
	\begin{align*}
		(m_{x_1}+ H_x)\otimes (a_1m_{y_1}+H_y)+ \cdots + (m_{x_n}+ H_x)\otimes (a_nm_{y_n}+H_y) = 0.
	\end{align*}
	First, assume that $m_{x_1}+ H_x,\dots, m_{x_n}+ H_x$ are distinct.
	If $m_{x_1}+ H_x,\dots, m_{x_n}+ H_x$ are linearly independent then Lemma \ref{Tensor_Linear_Independence} implies $a_1m_{y_1}+H_y = \cdots = a_nm_{y_n}+H_y =0$,
	which implies that $a_1=\cdots = a_n = 0$ since $\mathscr{M}_{S_y}$ is level linearly independent.
	Assume for the purposes of a contradiction that $m_{x_1}+ H_x,\dots, m_{x_n}+ H_x$ are linearly dependent.
	Then there are $b_1,\dots , b_n\in K$, not all $0$, such that for $f=b_1m_{x_1}+\cdots + b_nm_{x_n}$ we have $f+H_x=0$.
	Note that the homogeneous components of $f$ are in $H_x$ by Lemma \ref{Properties_of_Homogeneous_Ideals}.
	But, this contradicts level linear independence of $S_x$.
	
	Next, assume that $m_{x_1}+ H_x,\dots, m_{x_n}+ H_x$ are not distinct.
	Then we can rewrite the sum
	\begin{align*}
		(m_{x_1}+ H_x)\otimes (a_1m_{y_1}+H_y)+ \cdots + (m_{x_n}+ H_x)\otimes (a_nm_{y_n}+H_y) = 0,
	\end{align*}
	by grouping terms with the same $m_{x_i}+ H_x$.
	The proof from here follows similarly as in the distinct case by using Lemma \ref{Tensor_Linear_Independence} and level linear independence of $\mathscr{M}_{S_x}$ and $\mathscr{M}_{S_y}$.
	Therefore, $S_{xy}$ is level linearly independent.
	
	Finally, the existence of a monomial order follows easily from the poset isomorphism.
	In particular, $\mathscr{M}_{S_{xy}}$ is isomorphic to a product of two posets.
	If each of these posets has a total order on them then we can induce the lexicographic order on $\mathscr{M}_{S_{xy}}$.
	It is easily seen that if each of the total orders is a monomial order that the lexicographic order on $\mathscr{M}_{S_{xy}}$ will be a monomial order as well.
\end{proof}

Note that the linear independence condition in Lemma \ref{tensorsBasic} is a weaker version of level linear independence.
Lemma \ref{tensorsBasic} also gives us the following very useful result.

\begin{thm}[The Cartesian and Tensor Correspondence]\label{The_Cartesian_and_Tensor_Correspondence}
	Suppose that for all $i\in [d]$ we have $S_i = R_i/H_i$ for some homogeneous ideal $H_i$ of $R_i=K[x_{i,1},\dots, x_{i,n_i}]$,
	and that $S_i$ is level linearly independent for all $i\in [d]$.
	Let
	\begin{align*}
		S &=  \frac{K[x_{1,1},\dots, x_{1,n_1}, \dots , x_{d,1},\dots, x_{d,n_d} ]}{(H_1 + H_2 + \cdots + H_d)}.
	\end{align*}	
	Then
	\begin{enumerate}
		\item $S_1\otimes_K \cdots \otimes_K S_d \iso S$.
		\item $\mathscr{M}_{S_1}\times \cdots \times \mathscr{M}_{S_d} \iso \times \mathscr{M}_{S}$.
		\item $\mathscr{M}_{S}$ is level linearly independent.
		\item If there exist monomial orders for $\mathscr{M}_{S_1},\dots, \mathscr{M}_{S_d}$ then there is a monomial order on $\mathscr{M}_S$.
	\end{enumerate}
\end{thm}
\begin{proof}
	All the claims follow from induction, Lemma \ref{tensorsBasic}, and associativity of tensor products.
\end{proof}

We will have results that involve products of copies of a single poset or ring.
For this reason we make the following definition to distinguish between the operations.

\begin{dfn}[Cartesian and Tensor Powers]\label{Cartesian_and_Tensor_Powers}
	For a poset $\mathscr{P}$ and a ring $S$ we define the \textit{$n$-th Cartesian power} of $\mathscr{P}$ and the \textit{$n$-th tensor power} of $S$ to be
	\begin{align*}
		\mathscr{P}^{\times, n} &= \mathscr{P}\times \cdots \times \mathscr{P},\\
		S^{\otimes,n} &= S\otimes \cdots \otimes S.
	\end{align*}
\end{dfn}  \section{Advanced Orders}\label{Advanced_Orders}
In this section we define general classes of orders that are Macaulay orders for posets in later sections.

\subsection{The Hyperrectangle Chaser Orders}\label{The_Hyperrectangle_Chaser_Orders}

Informally, a hyperrectangle chaser order is a total order on a multiset lattices that prioritizes making a hyperrectangle of the highest dimension possible.
Every initial segment of a hyperrectangle chaser order is a downset.
A hyperrectangle chaser order can be seen in Figure \ref{Hyperrectangle_Chaser_Example}.

\begin{figure}
	\centering
	\begin{subfigure}[t]{0.16\textwidth}
		\includegraphics[width=\textwidth]{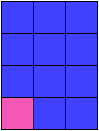}
	\end{subfigure}
	\hfill
	\begin{subfigure}[t]{0.16\textwidth}
		\includegraphics[width=\textwidth]{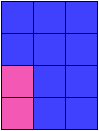}
	\end{subfigure}
	\hfill
	\begin{subfigure}[t]{0.16\textwidth}
		\includegraphics[width=\textwidth]{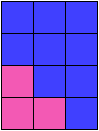}
	\end{subfigure}
	\hfill
	\begin{subfigure}[t]{0.16\textwidth}
		\includegraphics[width=\textwidth]{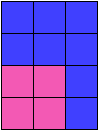}
	\end{subfigure}
	\hfill
	\begin{subfigure}[t]{0.16\textwidth}
		\includegraphics[width=\textwidth]{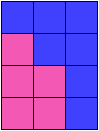}
	\end{subfigure}
	\hfill
	\begin{subfigure}[t]{0.16\textwidth}
		\includegraphics[width=\textwidth]{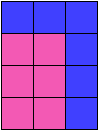}
	\end{subfigure}

	\vspace{0.5cm}
	\begin{subfigure}[t]{0.16\textwidth}
		\includegraphics[width=\textwidth]{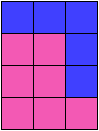}
	\end{subfigure}
	\hfill
	\begin{subfigure}[t]{0.16\textwidth}
		\includegraphics[width=\textwidth]{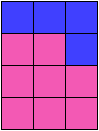}
	\end{subfigure}
	\hfill
	\begin{subfigure}[t]{0.16\textwidth}
		\includegraphics[width=\textwidth]{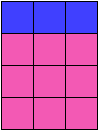}
	\end{subfigure}
	\hfill
	\begin{subfigure}[t]{0.16\textwidth}
		\includegraphics[width=\textwidth]{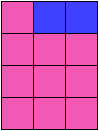}
	\end{subfigure}
	\hfill
	\begin{subfigure}[t]{0.16\textwidth}
		\includegraphics[width=\textwidth]{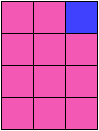}
	\end{subfigure}
	\hfill
	\begin{subfigure}[t]{0.16\textwidth}
		\includegraphics[width=\textwidth]{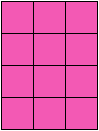}
	\end{subfigure}
	\caption{A hyperrectangle chaser order on $\mathscr{M}_{[2]}(3,4)$.}
	\label{Hyperrectangle_Chaser_Example}
\end{figure}

The reader might find reviewing Definition \ref{Indices_Elements_And_Intervals_Definition} and Definition \ref{Subproducts_Definition} before reading Definition \ref{The_Hyperrectangle_Chaser_Orders_Definition}.

\begin{dfn}[The Hyperrectangle Chaser Orders]\label{The_Hyperrectangle_Chaser_Orders_Definition}
	Let $d\geq 1$, consider finite tosets $\mathscr{T}_1,\dots, \mathscr{T}_d$ and put $\mathscr{T}=\mathscr{T}_1\times \cdots \times \mathscr{T}_d$.
	For $x=(x_1,\dots, x_d)\in \mathscr{T}$ we define the \textit{single coordinate distance} of $x$ from the origin to be
	\begin{align*}
		\scd(x) = \max_{i\in [d]} \mathscr{T}_i(x_i).
	\end{align*}
	Next, define the \textit{initial complement $\scd$} of $x$ to be the element that replaces any entry of $x$ whose index is not equal to $\scd(x)$ with the first element of the corresponding toset,
	\begin{align*}
		\ICscd(x) = (x_1',\dots, x_d'),
	\end{align*}
	where $x_i' = x_i$ if $\mathscr{T}_i(x_i)=\scd(x)$,
	and $x_i' = \mathscr{T}^{-1}(1)$ if $\mathscr{T}_i(x_i)\neq \scd(x)$.
	Finally we define a \textit{hyperrectangle chaser} order $\mathcal{HC}_{\mathscr{T}}$.
	We do this inductively, such that for each $S\subseteq [d]$ we define $\mathcal{HC}_S$ on $\mathscr{T}_S$, and of course $\mathcal{HC}_{\mathscr{T}} = \mathcal{HC}_{[d]}$.
	If $|S|=1$ then we define $\mathcal{HC}_S$ to be the original total order on $\mathscr{T}_S=\mathscr{T}_i$, where $i\in S\subseteq [d]$.
	So, suppose that $|S|>1$ and that $\mathcal{HC}_{S'}$ is defined for all $S'\subseteq [d]$ with $|S'| < |S|$.
	Furthermore, pick $\pi \in \mathfrak{S}_{|S|}$ and consider the domination order $\mathcal{D}=\mathcal{D}_{\pi}$.
	For $x,y\in \mathscr{T}_S$ we say that $x <_{\mathcal{HC}_S} y$ if one of the following conditions holds:
	\begin{enumerate}
		\item $\scd(x) < \scd(y)$
		\item $\scd(x) = \scd(y)$ and $\ICscd(x) <_{\mathcal{D}} \ICscd(y)$
		\item $\scd(x) = \scd(y) > 2$, $\ICscd(x) = \ICscd(y)$ and $x'<_{\mathcal{HC}_{S'}} y'$,
		where $S'$ is the set of coordinate for which the entries of $x$ and $y$ do not have index $\scd(x) = \scd(y)$,
		and $x'$ and $y'$ are obtained from $x$ and $y$ by deleting all entries equal to $\scd(x)$.
	\end{enumerate}
\end{dfn}

\begin{rem}
	Notice that there are many hyperrectangle chaser orders for one multiset lattice.
	A hyperrectangle chaser order depends on all the domination orders chosen for each subproduct.
\end{rem}

\begin{dfn}[The Lexicographic Hyperrectangle Chaser Order]\label{The_Lexicographic_Hyperrectangle_Chaser_Order}
	The \textit{lexicographic hyperrectangle chaser order} is denoted by $\mathcal{HC}_{\mathcal{L}}$, 
	and is the hyperrectangle chaser order for which the domination order chosen for every subproduct is always the lexicographic order.
\end{dfn}

Of course, when talking about the lexicographic hyperrectangle chaser order one needs to specify the multiset lattice.

It is important to note that hyperrectangle chaser orders appear in different discrete extremal problems, not just in the study of Macaulay posets.
Bollobás and Leader used them while studying edge isoperimetric problems in the grid \cite{Bollobs1991EdgeisoperimetricII}.
Ahlswede and Bezrukov used them when they extended the results of Bollobás and Leader to the arbitrary product of trees \cite{AhlswedeR.1995Eitf}. \subsection{The Border Chaser Orders}\label{The_Border_Chaser_Orders}

Informally, 
a border chaser order is a total order on a multiset lattice that prioritizes getting to the border,
where the border of a multiset lattice of dimension $d$ is all the points that have an upper shadow of less than $d$ points.
Every initial segment of a border chaser order is a downset.
A border chaser order can be seen in Figure \ref{border_Chaser_Example}.

\begin{figure}
	\centering
	\begin{subfigure}[t]{0.16\textwidth}
		\includegraphics[width=\textwidth]{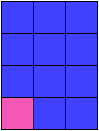}
	\end{subfigure}
	\hfill
	\begin{subfigure}[t]{0.16\textwidth}
		\includegraphics[width=\textwidth]{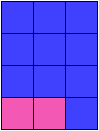}
	\end{subfigure}
	\hfill
	\begin{subfigure}[t]{0.16\textwidth}
		\includegraphics[width=\textwidth]{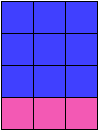}
	\end{subfigure}
	\hfill
	\begin{subfigure}[t]{0.16\textwidth}
		\includegraphics[width=\textwidth]{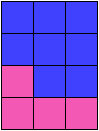}
	\end{subfigure}
	\hfill
	\begin{subfigure}[t]{0.16\textwidth}
		\includegraphics[width=\textwidth]{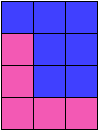}
	\end{subfigure}
	\hfill
	\begin{subfigure}[t]{0.16\textwidth}
		\includegraphics[width=\textwidth]{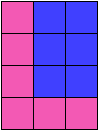}
	\end{subfigure}

	\vspace{0.5cm}
	\begin{subfigure}[t]{0.16\textwidth}
		\includegraphics[width=\textwidth]{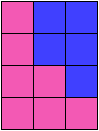}
	\end{subfigure}
	\hfill
	\begin{subfigure}[t]{0.16\textwidth}
		\includegraphics[width=\textwidth]{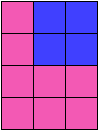}
	\end{subfigure}
	\hfill
	\begin{subfigure}[t]{0.16\textwidth}
		\includegraphics[width=\textwidth]{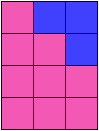}
	\end{subfigure}
	\hfill
	\begin{subfigure}[t]{0.16\textwidth}
		\includegraphics[width=\textwidth]{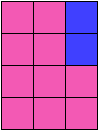}
	\end{subfigure}
	\hfill
	\begin{subfigure}[t]{0.16\textwidth}
		\includegraphics[width=\textwidth]{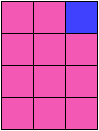}
	\end{subfigure}
	\hfill
	\begin{subfigure}[t]{0.16\textwidth}
		\includegraphics[width=\textwidth]{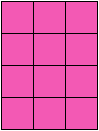}
	\end{subfigure}
	\caption{A border chaser order on $\mathscr{M}_{[2]}(3,4)$.}
	\label{border_Chaser_Example}
\end{figure}

Notice that a hyperrectangle chaser order prioritizes staying away from the border.
That is, a border chaser order is the dual of a hyperrectangle chaser order.
Let us be a little bit more precise.

\begin{dfn}[The Border Chaser Orders]\label{The_Border_Chaser_Orders_Definition}
	First, we consider the finite multiset lattice $\mathscr{M} = \mathscr{M}_{[d]}(\ell_1,\dots, \ell_d)$.
	Then $\mathscr{M}\iso \mathscr{M}^\ast$ by using the isomorphism of posets $\sigma$,
	that sends $(x_1,\dots, x_d)$ to $(\ell_1-x_1,\dots, \ell_d - x_d)$.
	Furthermore, notice that we have an isomorphism of tosets $(\mathscr{M}, \mathcal{D}) \iso (\mathscr{M}^\ast, \mathcal{D}^\ast)$ by $\sigma$,
	where $\mathcal{D}$ is a domination order on $\mathscr{M}$.
	
	Thus, for any hyperrectangle chaser order $\mathcal{HC}_{\mathscr{M}}$,
	the isomorphism $\sigma$ and $\mathcal{HC}_{\mathscr{M}}^\ast$ induce a total order on $\mathscr{M}$,
	such that $\sigma$ remains an isomorphism of tosets.
	We call this order the \textit{border chaser order} induced from $\mathcal{HC}_{\mathscr{M}}$ and denote it by $\mathcal{BC}_{\mathscr{M}}$.
\end{dfn}

\begin{rem}
	There are many border chaser orders on a multiset lattices, since there are many hyperrectangle chaser orders.
	A border chaser order depends on the choices for domination orders made for the corresponding hyperrectangle chaser order.
\end{rem}

\begin{dfn}[The Lexicographic border Chaser Order]\label{The_border_Hyperrectangle_Chaser_Order}
	The \textit{lexicographic border chaser order} is denoted by $\mathcal{BC}_{\mathcal{L}}$, and is the border chaser order induced from $\mathcal{HC}_{\mathcal{L}}$.
\end{dfn}
 \subsection{Partitioned Tosets and Block Orders}\label{Partitioned_Tosets}

Block orders are orders on multiset lattices,
where we first partition the multiset lattice,
then totally order the elements in each partition,
and finally we order the partitions.
An example of a block order can be seen in Figure \ref{Block_Order_Example}.

\begin{figure}
	\centering
	\begin{subfigure}[t]{0.2\textwidth}
		\includegraphics[width=\textwidth]{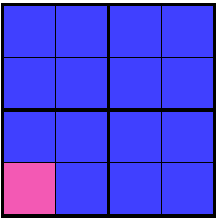}
	\end{subfigure}
	\hfill
	\begin{subfigure}[t]{0.2\textwidth}
		\includegraphics[width=\textwidth]{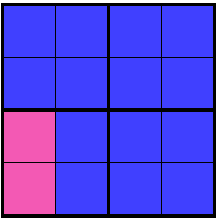}
	\end{subfigure}
	\hfill
	\begin{subfigure}[t]{0.2\textwidth}
		\includegraphics[width=\textwidth]{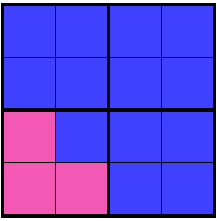}
	\end{subfigure}
	\hfill
	\begin{subfigure}[t]{0.2\textwidth}
		\includegraphics[width=\textwidth]{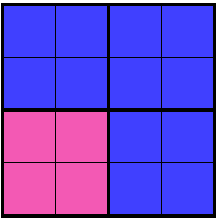}
	\end{subfigure}
	\hfill

	\vspace{0.5cm}
	\begin{subfigure}[t]{0.2\textwidth}
		\includegraphics[width=\textwidth]{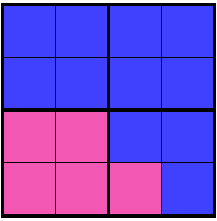}
	\end{subfigure}
	\hfill
	\begin{subfigure}[t]{0.2\textwidth}
		\includegraphics[width=\textwidth]{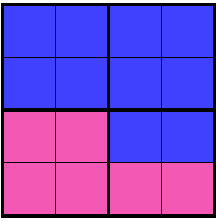}
	\end{subfigure}
	\hfill
	\begin{subfigure}[t]{0.2\textwidth}
		\includegraphics[width=\textwidth]{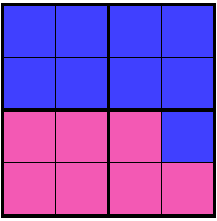}
	\end{subfigure}
	\hfill
	\begin{subfigure}[t]{0.2\textwidth}
		\includegraphics[width=\textwidth]{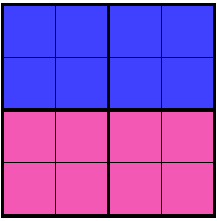}
	\end{subfigure}
	\hfill

	\vspace{0.5cm}
	\begin{subfigure}[t]{0.2\textwidth}
		\includegraphics[width=\textwidth]{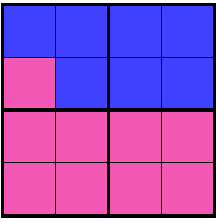}
	\end{subfigure}
	\hfill
	\begin{subfigure}[t]{0.2\textwidth}
		\includegraphics[width=\textwidth]{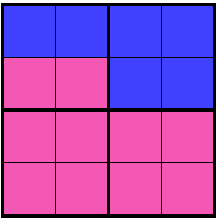}
	\end{subfigure}
	\hfill
	\begin{subfigure}[t]{0.2\textwidth}
		\includegraphics[width=\textwidth]{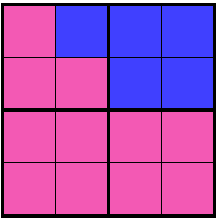}
	\end{subfigure}
	\hfill
	\begin{subfigure}[t]{0.2\textwidth}
		\includegraphics[width=\textwidth]{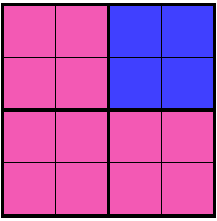}
	\end{subfigure}
	\hfill

	\vspace{0.5cm}
	\begin{subfigure}[t]{0.2\textwidth}
		\includegraphics[width=\textwidth]{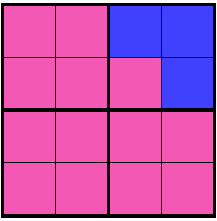}
	\end{subfigure}
	\hfill
	\begin{subfigure}[t]{0.2\textwidth}
		\includegraphics[width=\textwidth]{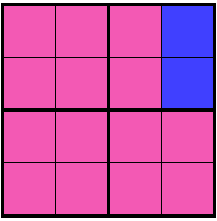}
	\end{subfigure}
	\hfill
	\begin{subfigure}[t]{0.2\textwidth}
		\includegraphics[width=\textwidth]{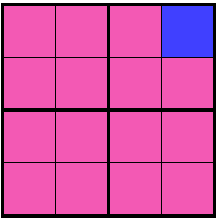}
	\end{subfigure}
	\hfill
	\begin{subfigure}[t]{0.2\textwidth}
		\includegraphics[width=\textwidth]{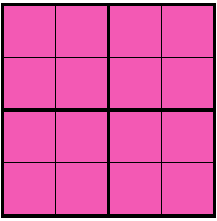}
	\end{subfigure}
	\hfill
	\caption{A block order on $\mathscr{M}_{[2]}(4,4)$, where each of the tosets is partitioned into $2$ parts.}
	\label{Block_Order_Example}
\end{figure}

\begin{dfn}[Ordered Partitions and Partitioned Tosets]\label{Ordered_Partitions_and_Partitioned_Tosets_Definition}
	Suppose that $\mathscr{T}$ is a finite toset of size $n$.
	We say that a partition $P$ of $\mathscr{T}$ is an \textit{ordered partition} if there are integers $1= a_1<a_2 \cdots < a_{k-1}<a_k=n$ such that 
	\begin{align*}
		P &= \{\mathscr{T}[a_i,a_{i+1}) \bigm | i\in [k-1]\}.
	\end{align*}
	We call $\mathscr{T}$ a \textit{partitioned toset} with $P$.
	The set $\Starts_P = \{\mathscr{T}^{-1}[a_i]\bigm | i\in [k-1]\}$ is called the \textit{set of starts of $P$}. 
\end{dfn}

Next, we extend the ideas in Definition \ref{Ordered_Partitions_and_Partitioned_Tosets_Definition} to Cartesian products.

\begin{dfn}[Products of Partitioned Tosets]\label{Products_of_Partitioned_Tosets}
	Suppose that we have finite tosets $\mathscr{T}_1,\dots, \mathscr{T}_d$ with corresponding ordered partitions $P_1,\dots, P_d$.
	Then we can consider the products 
	\begin{align*}
		\mathscr{T} &=\mathscr{T}_d\times \cdots \times \mathscr{T}_d,\\
		P &= P_1\times \cdots \times P_d,\\
		\Starts_P &= \Starts_{P_1}\times \cdots \times \Starts_{P_d}.
	\end{align*}
	We call the elements of $P$ the \textit{blocks} of $\mathscr{T}$ and the elements of $\Starts_P$ the \textit{starts} of $\mathscr{T}$.
	Clearly, $P$ and $\Starts_P$ are in a bijective correspondence. 
	For $B\in P$ we define $\start_P(B)$ to be the corresponding element in $\Starts_P$ under this bijective correspondence,
	and for $s\in \Starts_P$ we define $\Block_P(s)$ to be the corresponding element in $P$ under this bijective correspondence.
	Notice that $P$ is a partition for $\mathscr{T}$.
	Thus, for every $x\in \mathscr{T}$ we write $\Block_P(x)$ for the corresponding block that $x$ belongs to.
	Furthermore, we define the start of $x$ to be the start of the block that $x$ belongs to, $\start_P(x) = \start_P(\Block_P(x))$.
	We omit the subscript $P$ when it is clear from context.
	
	We know that $\mathscr{T}$ is a poset and we are going to make it into a toset now.
	First, we pick a total order $\mathcal{O}_{\Starts}$ on $\Starts$.
	Second, for each $B\in P$ we pick a total order $\mathcal{O}_B$.
	Then we define the total order $\mathcal{O} = (\mathcal{O}_B)_{B\in P}\mathcal{O}_{\Starts}$, 
	such that for $x,y\in \mathscr{T}$ we have $x<_{\mathcal{O}} y$ iff
	\begin{enumerate}
		\item $\start(x) <_{\mathcal{O}_{\Starts}} \start(y)$, or
		\item $\start(x) = \start(y)$ and $x <_{\mathcal{O}_{\Block(x)}} y$.
	\end{enumerate}
	The order $\mathcal{O}$ is called a \textit{block} order.
	It is clear that $\mathcal{O}$ and $P$ make $\mathscr{T}$ into a partitioned toset.
\end{dfn}  \section{Quotients by Monomial Ideals}\label{Quotients_by_Monomial_Ideals}

Monomial ideals have gathered the most attention when it comes to analogs of Macaulay theorem stated in terms of rings.
We start by showing that the Mermin--Murai Theorem on colored square-free rings follows from the Macaulay theorem on star posets.
After that we provide new results that were unknown to either algebraists or combinatorialist.

\numberwithin{thm}{subsection}
\subsection{Star Posets: A New Proof of the Mermin--Murai Theorem}\label{stars}

\begin{dfn}[Star Posets]
	A \textit{basic star poset} is a set of $n+1$ elements with $n\in \N\setminus \{0\}$,
	such that it forms a ranked poset that has two levels,
	where $|\Lvl_{0}| = n$, $|\Lvl_1| = 1$ and every element in $\Lvl_0$ is less than the unique element in $\Lvl_1$.
	We write $\Star(n)$ for the basic star with $n+1$ elements.
	Then a \textit{star poset} is the Cartesian product of basic star posets,
	and we denote it by $\Star(n_1,\dots, n_d) = \Star(n_1)\times \cdots \times \Star(n_d)$.
\end{dfn}

We will often say that the basic star poset $\Star(n)$ has $n$ legs.
Note that $\Star(1) = \mathscr{M}_{[1]}(2)$.
Another way to think of $\Star(n)$ is to take $n$ copies of $\mathscr{M}_{[1]}(2)$ and glue the largest elements together.
Some Hasse graphs of basic star posets are shown in Figure \ref{Hasse_Graph_Basic_Star_Example}.

\begin{figure}
	\centering
	\hfill
	\begin{subfigure}[t]{0.0145\textwidth}
		\includegraphics[width=\textwidth]{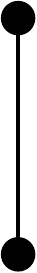}
	\end{subfigure}
	\hfill
	\begin{subfigure}[t]{0.10667\textwidth}
		\includegraphics[width=\textwidth]{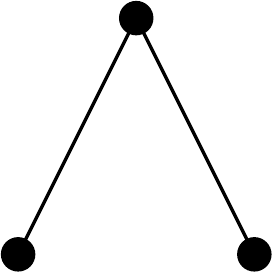}
	\end{subfigure}
	\hfill
	\begin{subfigure}[t]{0.21333\textwidth}
		\includegraphics[width=\textwidth]{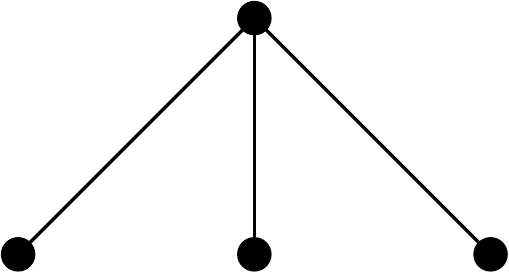}
	\end{subfigure}
	\hfill
	\begin{subfigure}[t]{0.32\textwidth}
		\includegraphics[width=\textwidth]{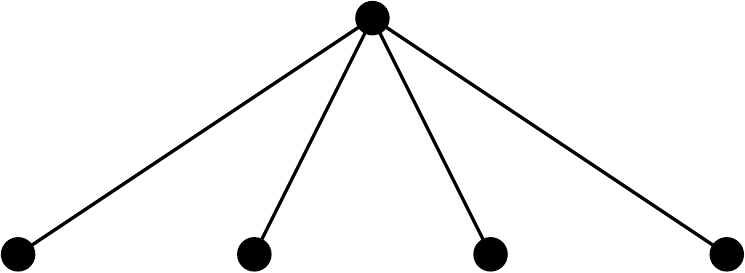}
	\end{subfigure}
	
	\caption{Hasse graphs of some basic star posets.}\label{Hasse_Graph_Basic_Star_Example}
\end{figure}

\begin{thm}[Star Macaulay Theorem]\label{starTheorem}
	All star posets are Macaulay.
\end{thm}

Bezrukov's Dual Lemma gives us a corollary right away.

\begin{cor}[Star Dual Macaulay Theorem]\label{starDual}
	The dual of any star poset is Macaulay.
\end{cor}

A proof of Theorem \ref{starTheorem} can be found in \cite{engel_1997}.
Many authors contributed to the proof of the star Macaulay Theorem over a long period of time.
Lindström \cite{LindstromBernt1971Tono} proved the claim for $\Star(2,\dots, 2)$.
Leeb \cite{leeb1978salami} handled the case $\Star(n,\dots, n)$ for any $n\in \N$ with $n\geq 2$.
Bezrukov handled the same case as Leeb independently in \cite{BezrukovStars}.
Leeb in \cite{leeb1978salami} mentioned that the result can be extended to all stars,
and Leck provided proof of this in \cite{leck1995extremalprobleme}.

Bollobás and Radcliffe studied face isoperimetric inequalities of $\mathscr{M}_{[d]}(2,\dots, 2)$ in \cite{BollobasB.1990IIfF}.
These results were later extended to face isoperimetric inequalities in finite multiset lattices by Bollobás and Leader in \cite{BollobasBela1990EFI}.
Bezrukov in \cite{bezrukovDual} noticed that these face isoperimetric inequalities are special cases of the Macaulay poset problem on stars.
Engel has a nice exposition of this in \cite{engel_1997}.

Frankl, F{\"u}redi and Kalai studied shadows of colored complexes in \cite{FRANKLPETER1988SOCC}.
This result came to be known as the colored Kruskal--Katona Theorem.
A simpler proof of this result was given by London in \cite{LondonEran1994Anpo}.
Engel in \cite{engel_1997} noticed that colored complexes are isomorphic to the duals of star posets, 
and generalized the Frankl--F{\"u}redi--Kalai Theorem by using Corollary \ref{starDual}.

It is clear that many authors found the star posets and their duals to be important.
All of the work mentioned so far was done before the 21st century.
In more recent times, Mermin and Murai in \cite{MerminJeff2010Bnol} introduced colored square-free rings and studied Hilbert functions of homogeneous ideals in them.
In particular, they proved that colored square-free rings are Macaulay.
This result was inspired by the Frankl--F{\"u}redi--Kalai Theorem.
In fact, the Mermin--Murai Theorem generalizes the Frankl--F{\"u}redi--Kalai Theorem.
Murai used the Frankl--F{\"u}redi--Kalai Theorem to study Koszul toric rings \cite{MURAISATOSHI2011Frol}.
It turns out that we can easily prove the Mermin--Murai Theorem by applying Engel's generalization of the Frankl--F{\"u}redi--Kalai Theorem and The Macaulay Correspondence Theorem \ref{Macaulay_Correspondence_Theorem}.

\begin{dfn}[Power and Disjoint Product Ideals]\label{Power_and_Disjoint_Product_Ideals}
	For $R=K[x_1,\dots, x_d]$ and $n_1,\dots, n_d\in \N\cup\{\infty\}$,
	we define the \textit{power ideal of $R$ and $(n_1,\dots, n_d)$} to be ideal generated by $x_1^{n_1},\dots, x_d^{n_d}$,
	\begin{align*}
		\Pow_R(n_1,\dots, n_d) = (x_1^{n_1},\dots, x_d^{n_d}).
	\end{align*}
	For an integer $n\geq 2$ we define the \textit{disjoint variable product ideal of $R$ and $n$} to be
	\begin{align*}
		\DVP_{R}(n) = \left( \left\lbrace \prod_{i\in A} x_i \bigm | A\subseteq [d] \textit{ and } |A| = n \right\rbrace\right).
	\end{align*}
	Alexandra Seceleanu pointed out to the author that $\DVP_{R}(n)$ is called the \textit{$n$-th square-free Veronese ideal}.
\end{dfn}

\begin{dfn}[Colored Square-Free Rings]
	Let $R = K[x_{1},\dots, x_{d}]$ and define the homogeneous ideal $H=\Pow_R(2,\dots, 2) + \DVP_{R}(2)$.
	So, $H=(x_{1},\dots, x_{d})^2$.
	We call $S=R/H$ a \textit{basic colored square-free ring}.
	A \textit{colored square-free ring} is a tensor product of basic colored square-free rings.
\end{dfn}

\begin{cor}[Mermin--Murai \cite{MerminJeff2010Bnol} 2010]\label{MerminMurai}
	All colored square-free rings are Macaulay.
\end{cor}
\begin{proof}
	Note that the poset of monomials of a basic colored square-free ring is the dual of a basic star poset (see Figure \ref{Star_to_Color}).
	The claim now follows from Theorem \ref{The_Cartesian_and_Tensor_Correspondence}, Proposition \ref{Properties_of_Duals}, Corollary \ref{starDual} and the Macaulay Correspondence Theorem \ref{Macaulay_Correspondence_Theorem}.
\end{proof}

\begin{figure}
	\centering
	\hfill
	\begin{subfigure}[t]{0.2\textwidth}
		\includegraphics[width=\textwidth]{Star_legs_3.pdf}
	\end{subfigure}
	\hfill
	\begin{subfigure}[t]{0.18\textwidth}
		\raisebox{0.25\hsize}{\includegraphics[width=\textwidth]{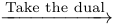}}
	\end{subfigure}
	\hfill
	\begin{subfigure}[t]{0.2\textwidth}
		\includegraphics[width=\textwidth]{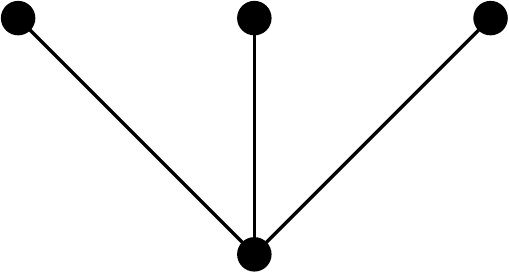}
	\end{subfigure}
	\begin{subfigure}[t]{0.18\textwidth}
		\raisebox{0.25\hsize}{\includegraphics[width=\textwidth]{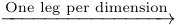}}
	\end{subfigure}
	\hfill
	\begin{subfigure}[t]{0.2\textwidth}
		\includegraphics[width=\textwidth]{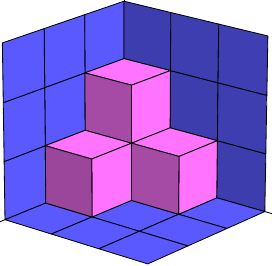}
	\end{subfigure}
	\hfill
	
	\caption{Relationship between star posets and colored square-free rings.}\label{Star_to_Color}
\end{figure}

The only thing left to do is describe a Macaulay order for the above posets.
We describe the Macaulay order from \cite{MerminJeff2010Bnol} and show that it is equivalent to a block order.
It turns out that all the Macaulay orders we are using in this paper are block orders,
and these types of orders allow for some of the most general Macaulay theorems available to us at the moment.

Suppose that we have colored square-free rings $S_1,\dots, S_d$,
such that $S_i$ has $n_i$ variable and $n_1\geq n_2\geq \cdots \ge n_d$.
Let $\mathscr{M}_{S_i} = \{1,x_{i,1},\cdots , x_{i, n_i}\}$.
Then we define an order on the variables such that $x_{k,\ell} > x_{k, \ell'}$ if $\ell > \ell'$ or $\ell = \ell'$ and $k<k$.
This ordering of the variables generates a permutation and hence a domination order on the isomorphic ring to $S_1\otimes \cdots \otimes S_d$ given by Theorem \ref{The_Cartesian_and_Tensor_Correspondence}.
A proof that this order is Macaulay order is given by Mermin and Murai in \cite{MerminJeff2010Bnol}.
We will call this order the Mermin--Murai order.

We now describe a block order on $S_1,\dots, S_d$.
We abuse notation and write $\mathscr{M}_{S_i} = \{1<x_{i,1}<\cdots < x_{i, n_i}\}$ to form a toset out of each poset of monomials.
Now we partition each of these tosets such that $\{1, x_{i,1}\}$ is the first partition and then every other variable is in a partition with itself.
The block order we want on the product $S_1\otimes\cdots \otimes S_d$ is formed by using the lexicographic hyperrectangle chaser order to order the starts,
and we use the lexicographic order to order each of the individual blocks.

The lexicographic hyperrectangle chaser order on the starts gives us that monomials with variables $x_{i,j}$ with smaller $j$ will come earlier in the block order.
The condition that we order each block with the lexicographic order gives us that monomials (in the same block) with variables $x_{i,j}$ with larger $i$ will come earlier in the block order.
These two conditions gives us equality between the block order and the Mermin--Murai domination order.

 \subsection{Spider Posets: Extending the Mermin--Murai Theorem}\label{Spider_Posets}

Spider posets are generalizations of stars,
where we extend each of the legs by the same length.
Another way to think about spider posets is to take copies of $\mathscr{M}_{[1]}(\ell)$ and glue the largest elements together.
The Hasse graphs of some spider posets are given in Figure \ref{Hasse_Graphs_of_some_spider_posets},
and a formal definition of spider posets is given in Definition \ref{Spider_Posets_Definition}.

\begin{figure}
	\centering
	\hfill
	\begin{subfigure}[t]{0.014\textwidth}
		\includegraphics[width=\textwidth]{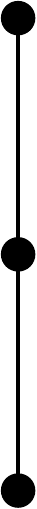}
	\end{subfigure}
	\hfill
	\begin{subfigure}[t]{0.10667\textwidth}
		\includegraphics[width=\textwidth]{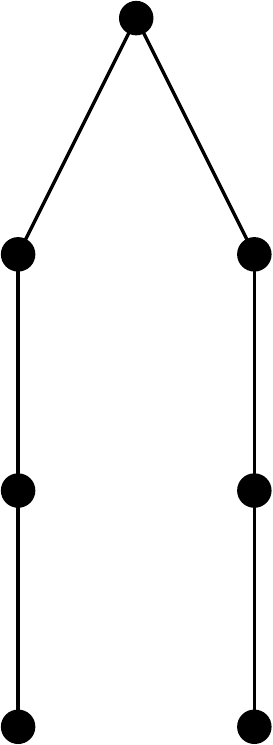}
	\end{subfigure}
	\hfill
	\begin{subfigure}[t]{0.21333\textwidth}
		\includegraphics[width=\textwidth]{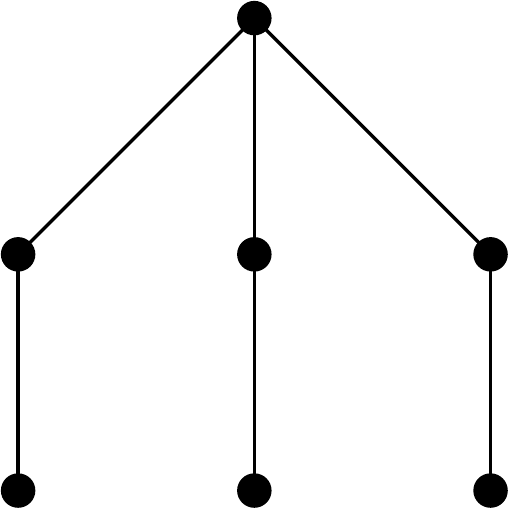}
	\end{subfigure}
	\hfill
	\begin{subfigure}[t]{0.32\textwidth}
		\includegraphics[width=\textwidth]{Star_legs_4.pdf}
	\end{subfigure}
	
	\caption{Hasse graphs of some spider posets.}\label{Hasse_Graphs_of_some_spider_posets}
\end{figure}

\begin{dfn}[Spider Posets]\label{Spider_Posets_Definition}
	For integers $k,\ell\in \N$ we define the \textit{spider set} $\Spider(k,\ell) = \{0,1,2,\dots (k+1)\ell\}$,
	and a total order $\mathcal{O}$ on $\Spider(k,\ell)$ such that for $a,b\in \Spider(k,\ell)$ we have $a\leq_{\mathcal{O}} b$ iff
	\begin{enumerate}
		\item $a \equiv b \pmod{(k+1)}$ and $a\leq b$, or
		\item $b = (k+1)\ell$.
	\end{enumerate}
	We say that $(\Spider(k,\ell), \mathcal{O})$ is a \textit{spider poset}.
	We will often abuse notation and refer to $(\Spider(k,\ell), \mathcal{O})$ just by the set $\Spider(k,\ell)$.
	We call $(k+1)\ell$ the \textit{head} of $\Spider(k,\ell)$.
	We say that $\Spider(k,\ell)$ has $k+1$ legs each of length $\ell$,
	where a \textit{leg} consists of all numbers that are not the head and are in the same equivalence class modulo $(k+1)$. 
\end{dfn}

Clearly $\Spider(k,2) \iso \Star(k+1)$.
Thus, spider posets are a natural generalization of basic star posets.
Spider posets were first considered by Bezrukov in \cite{BezrukovSpiders}, 
and then later by Bezrukov and Elsässer in \cite{BezrukovElsasserSpiders}.
In particular, the Cartesian power of a spider poset was considered in both cases.
For this reason we make the following definition.

\begin{dfn}[Bezrukov--Elsässer Posets]
	A \textit{Bezrukov--Elsässer poset} is a poset of the form $(\Spider(k,\ell))^{\times, n}$ for some $k,\ell,n\in \N$.
\end{dfn}

\begin{thm}[Bezrukov--Elsässer \cite{BezrukovElsasserSpiders} (2000)]\label{Spider_Macaulay_Theorem}
	All Bezrukov--Elsässer posets are Macaulay.
\end{thm}

Of course, Bezrukov's Dual Lemma \ref{Bezrukov_Dual_Lemma} yields a corollary right away.

\begin{cor}\label{Dual_Spider_Macaulay_Theorem}
	All duals of Bezrukov--Elsässer posets are Macaulay.
\end{cor}

The reader is probably anticipating at this point that we plan to use the same technique we used for stars and the Bezrukov--Elsässer Theorem \ref{Spider_Macaulay_Theorem} to produce a result on rings.
In honor of the founders of Theorem \ref{Spider_Macaulay_Theorem} we make the following definition.

\begin{dfn}[Bezrukov--Elsässer Rings]\label{Bezrukov_Elsässer_Ring_Definition}
	Let $R=K[x_1,\dots, x_d]$, choose positive integer $\ell \in \N$ and set $H=\Pow_R(\ell,\dots, \ell) + \DVP_{R}(2) = (x_1^\ell,\dots, x_d^\ell) + (x_ix_j \bigm | i <j)$.
	We call $R/H$ a \textit{basic Bezrukov--Elsässer ring},
	and for $n\in \N$ we call $(R/H)^{\otimes, n}$ a \textit{Bezrukov--Elsässer ring}.
\end{dfn}

From here we get a generalization of the Mermin--Murai Theorem \ref{MerminMurai} to rings which are not square-free.

\begin{thm}\label{Bezrukov_Elsässer_Macaulay_Theorem}
	All Bezrukov--Elsässer rings are Macaulay.
\end{thm}
\begin{proof}
	Note that the poset of monomials of a basic Bezrukov--Elsässer ring is the dual of a spider poset (see Figure \ref{Spider_to_Color}).
	The claim now follows from Theorem \ref{The_Cartesian_and_Tensor_Correspondence}, Proposition \ref{Properties_of_Duals}, Corollary \ref{Dual_Spider_Macaulay_Theorem} and the Macaulay Correspondence Theorem \ref{Macaulay_Correspondence_Theorem}.
\end{proof}

\begin{figure}
	\centering
	\hfill
	\begin{subfigure}[t]{0.2\textwidth}
		\includegraphics[width=\textwidth]{Spider_legs_3_length_2.pdf}
	\end{subfigure}
	\hfill
	\begin{subfigure}[t]{0.17\textwidth}
		\raisebox{0.45\hsize}{\includegraphics[width=\textwidth]{Arrow_Dual.pdf}}
	\end{subfigure}
	\hfill
	\begin{subfigure}[t]{0.2\textwidth}
		\includegraphics[width=\textwidth]{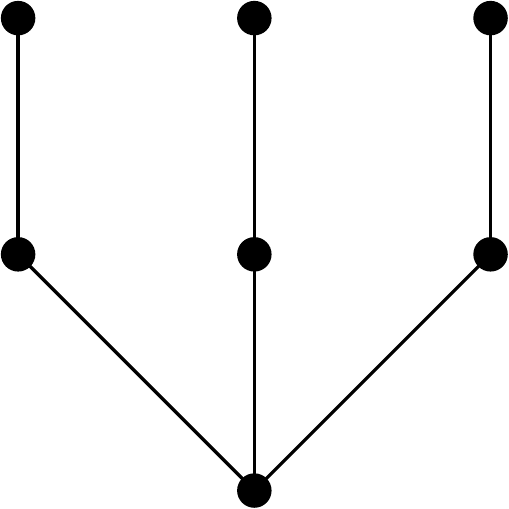}
	\end{subfigure}
	\begin{subfigure}[t]{0.17\textwidth}
		\raisebox{0.45\hsize}{\includegraphics[width=\textwidth]{Arrow_Dual_Euclid.pdf}}
	\end{subfigure}
	\hfill
	\begin{subfigure}[t]{0.22\textwidth}
		\includegraphics[width=\textwidth]{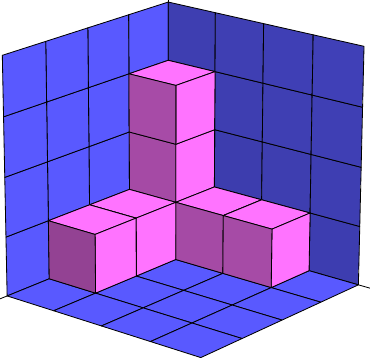}
	\end{subfigure}
	\hfill
	
	\caption{Relationship between spider posets and Bezrukov--Elsässer rings.}\label{Spider_to_Color}
\end{figure}

The only thing left to do is to specify a Macaulay order on Bezrukov--Elsässer posets.
If one starts with the simpler definition of the Macaulay order on stars \cite{Bezrukov2004,engel_1997,MerminJeff2010Bnol},
then it takes more effort to generalize to a Macaulay order on Bezrukov--Elsässer posets.
However, with the block orders the generalization is natural, and just focuses on grouping the legs of spider posets just like we did for basic star posets.
The main result in \cite{BezrukovElsasserSpiders} is the proof that the order below is Macaulay.

First, for every copy of $\Spider(k,\ell)$ in the product $\Spider(k,l)^{\times, n}$,
we form a partitioned toset.
The first part is made up of the leg that contains $0$,
Then the second part is made up of the leg that contains $1$, and so on,
until we get to the leg that contains $k$.
The last part is made up of the leg that contains $k$ together with the head of the spider poset.
Inside each part we order the elements by the standard order on $\N$, which is a restriction of the partial order on $\Spider(k,\ell)$.
Note that this is a similar partitioning like the one we did on the basic star posets, 
the legs were just shorter.

The starts of $\Spider(k,l)^{\times, n}$ are ordered by the lexicographic border chaser order.
Again, note that this is similar to the case for duals of basic stars. 
There we just considered the dual of the border chaser order, which is the hyperrectangle chaser order.

We have to be a little careful with the ordering in each block.
The order is not the lexicographic order, but a domination order.
To describe this order we will construct a permutation $\pi_B = (p_1,\dots ,p_n)$ for each block $B$.

Consider a block of $\Spider(k,\ell)^{\times, n}$, say $B=B_1\times \cdots \times B_{n}$.
Then each $B_i$ is a partition (leg) of $\Spider(k,\ell)$ that contains exactly one element from the set $\{0,1,\dots, k\}$.
Let $m_1$ be the largest integer in $(B_1\cup \cdots \cup B_n)\cap \{1,\dots, k\}$.
Then let $j_1< \cdots < j_{q_1}$ be such that $m\in B_{j_1},\dots, B_{j_q}$.
We define $ p_{n-q_1+1} = j_1,\dots, p_n=j_{q_1}$.
Let $m_2$ be the largest integer in $(B_1\cup \cdots \cup B_n)\cap (\{1,\dots, k\}\setminus \{m_1\})$.
Then let $i_1< \cdots < i_{q_2}$ be such that $m\in B_{i_1},\dots, B_{i_2}$.
We define $p_{n-q_1-q_2+1} = i_1,\dots, p_{n-q_1}=i_{q_2}$.
Continue like this and define $\pi_B = (p_1,\dots, p_n)$.
We order the block $B$ by the domination order $\mathcal{D}_{\pi_B}$.

With every block being ordered by a domination order, 
and all the starts being ordered by the lexicographic border chaser order,
we finally have a total order on $\Spider(k,\ell)^{\times, n}$.
It is not obvious that this order is Macaulay and a proof that it is Macaulay is given in \cite{BezrukovElsasserSpiders}.
The notation in \cite{BezrukovElsasserSpiders} is different from the one we have used here.
In \cite{BezrukovElsasserSpiders},
for $a\in \Spider^{\times, n}(k,\ell)$ the authors define $\underline{a}$ and $\vec{a}$,
and use these to define a total order.
The definition of $\underline{a}$ in \cite{BezrukovElsasserSpiders} in our notation means the start of the block that $a$ belongs to.
The definition of $\vec{a}$ in \cite{BezrukovElsasserSpiders} in our notation means $\pi_B(a)$, where $B$ is the block that $a$ belongs to.
These facts follow from the definitions. 
\subsection{Products With Multiset Lattices}\label{Products_With_Multiset_Lattices}

The oldest results on Macaulay posets and rings concern multiset lattices.
It is natural to ask about Macaulay theorems that contain multiset lattices.
This has been studied by several authors.
The first such result we are aware of, is by Bezrukov and Leck.

\begin{thm}[Bezrukov--Leck \cite{Bezrukov2004} 2004]
	Suppose that $\mathscr{P}$ is a ranked poset with $r(\mathscr{P}) = 1$ and let $q\in \N$ with $q\geq 1$.
	Then $\mathscr{P}\times \mathscr{M}_{[1]}(q)$ is Macaulay if and only if $\mathscr{P}$ is Macaulay.
\end{thm}

Unfortunately, the above result does not give us a lot of information when it comes to rings.
The only monomial posets of rank $1$ are just duals of stars.
So, this gives us that the tensor product of a basic colored square-free ring with $K[x]/(x^q)$ is Macaulay.

The result above of Bezrukov and Leck was inspired by work done by Clements \cite{ClementsG.F.1977Motg}, 
which characterized the Macualay property in terms of additivity,
for the product of two posets where one poset has rank $0$.
Bezrukov and Leck suggested that the next step to take is to characterize all Macaulay posets that are the product of a Macaulay poset with $\mathscr{M}_{[1]}(\ell)$ where $\ell\in \N$.
It turns out that such a result was first discovered by Mermin and Peeva, and independently by Shakin, in terms of Macaulay rings.
\begin{thm}[Mermin--Peeva \cite{MerminJeffrey2006Li} 2006, Shakin \cite{ShakinD.2007Mi} 2007]\label{Mermin_Peeva_and_Shakin}
	Suppose that $R=K[x_1,\dots, x_d]$ and $M$ is a monomial ideal of $R$.
	If $(R/M, \mathcal{L})$ is Macaulay then $(R[x_{d+1}]/(M'), \mathcal{L})$ is Macaulay,
	where $M'$ is the image of $M$ under the natural inclusion map $R\xhookrightarrow{} R[x_{d+1}]$. 
\end{thm}

Note that $R[x_{d+1}]/(M')\iso K[x_1,\dots, x_{d+1}]/(M) \iso R/M \otimes_K K[x_{d+1}]/(0)$.
Thus, if the monomial poset of $R/M$ is Macaulay with $\mathcal{L}$,
then the Cartesian product of the monomial poset of $R/M$ with $\mathscr{M}_{[1]}(\infty)$ is Macaulay.
This gives us an answer to the problem posed by Bezrukov and Leck.

\begin{cor}\label{Bezrukov_Leck_Chain}
	Suppose that $\mathscr{P}$ is a poset representable by a ring, where the quotient ideal is monomial.
	If $(\mathscr{P}, \mathcal{L})$ is Macaulay then $(\mathscr{P}\times \mathscr{M}_{[1]}(\infty),\mathcal{L})$ is Macaulay.
\end{cor}

The result of Mermin and Peeva, and Shakin, does not have a rank restriction, 
but it only works for a poset that is the poset of monomials of some quotient of a polynomial ring by a monomial ideal.
Chong \cite{ChongKaiFongErnest2015Hfoc} was studying generalizations of the Mermin--Murai Theorem \ref{MerminMurai} and discovered more results like the two previous theorems.
After translation, Chong's results have a very natural description in our framework.

\begin{thm}[Chong \cite{ChongKaiFongErnest2015Hfoc} 2015]\label{Chong_Macaulay_1}
	The tensor product of a colored square-free ring and a Clements--Lindström ring, is Macaulay.
	In the language of posets, the product of a star and a multiset lattice is Macaulay. 
\end{thm}

\begin{thm}[Chong \cite{ChongKaiFongErnest2015Hfoc} 2015]\label{Chong_Macaulay_2}
	Suppose that $S=K[x_1,\dots, x_d]/(x_1,\dots, x_d)^{a+1}$ for some integer $1\leq a \leq \infty$,
	and consider integers $\ell_{d'}\geq \ell_{d'-1}\geq \cdots \geq \ell_1 \geq a+1$ and positive $d'\in \N$.
	Then the tensor product of $S$ with a Clements--Lindström ring whose poset of monomials is isomorphic to $\mathscr{M}_{[d']}(\ell_1,\dots, \ell_{d'})$, 
	is Macaulay.
\end{thm}

Chong's theorem also give results to the problem posed by Bezrukov and Leck.
It should be mentioned that Chong's results in \cite{ChongKaiFongErnest2015Hfoc} are more general than what we have presented here.
A generalization of colored square-free rings is defined, 
and Chong shows that among these generalized colored square-free rings, the only Macaulay rings with a certain domination order, 
are the ones in Theorem \ref{Chong_Macaulay_1} and Theorem \ref{Chong_Macaulay_2}.
 \subsection{Posets of Monomials Whose Hasse Graphs are Trees}\label{Posets_of_Monomials_that_are_Trees}

There is a general theme to all the Macaulay rings presented in the previous sections.
They all involve taking the tensor product of a Macaulay ring with another Macaulay ring whose poset of monomials is a tree.
In fact, something much stronger is true.
Most of the rings can be decomposed into a tensor product of Macaulay rings whose posets of monomials are all trees.
The exceptions to this come from Section \ref{Products_With_Multiset_Lattices},
where one of the posets is allowed to not be a tree,
but again we are still taking the product of a general poset with a tree.
In this section, for a large class of rings whose posets are trees, we settle the problem of when the tensor product is Macaulay.

\begin{dfn}[Tree Rings]\label{Tree_Rings_Definition}
	Suppose that $H\subseteq K[x_1,\dots, x_d]$ is a homogeneous ideal.
	We say that $S=K[x_1,\dots, x_d]/H$ is a \textit{tree ring} if the Hasse graph of $\mathscr{M}_S$ is a tree.
\end{dfn}

The following result shows that if $S$ is a tree ring then the Hasse graph of $\mathscr{M}_S$ is very close to the dual of a spider poset.

\begin{lem}\label{Tree_Rings_Concrete}
	Suppose that $S=K[x_1,\dots, x_d]/H$ is a tree.
	There exists a set $A=\{i_1,\dots, i_n\}\subseteq [d]$ such that:
	\begin{enumerate}
		\item For all $i\in A$ we have $x_i+H\neq 0$.
		\item For all $i\in [d]\setminus A$ we have $x_i+H=0$.
		\item For all $i,j\in A$ with $i\neq j$ we have $x_i+H \neq x_j+H$.
		\item For all $i,j\in A$ with $i\neq j$ we have that $x_ix_j+H=0$.
		\item For all $m\in \mathscr{M}_S$ there exists $i\in A$ and $p\in \N$ such that $m=x_i^p+H$.
		\item If $i,j\in A$ such that $i\neq j$ and we take positive $p,q\in \N$ such that $x_i^p +H,x_j^q +H \neq 0$,
		then we have $x_i^p +H \neq x_j^q+H$.
	\end{enumerate}
	In particular, $\mathscr{M}_S$ is isomorphic to the poset of monomials of
	\begin{align*}
	\frac{K[x_{i_1},\dots, x_{i_n}]}{(x_1^{e_1},\dots, x_{i_n}^{e_n}) + (x_ix_j \bigm | i<j)},
	\end{align*}
	for some $e_1,\dots, e_n\geq 2$.
	Note that this tells us that $\mathscr{M}_S$ is isomorphic to a poset that is representable by a ring that is a quotient by a monomial ideal.
	Therefore, Lemma \ref{Monomial_Quotients_Produce_Subposets} guarantees that there is always a monomial order on $\mathscr{M}_S$.
\end{lem}
\begin{proof}
	First, we define a sequence of sets $A_1,A_2,A_3,\dots \subseteq [d]$ by induction.
	Let $A_1=\{i_1\}$, where $i_1\in [d]$ is the smallest integer such that $x_{i_{1}}+H \neq 0$.
	Suppose that $k\geq 2$ and that $A_{k'}$ is defined for all $1\leq k'<k$.
	We define $A_k$ now.
	Let $B_k$ be the set of all integers $i\in [d]$ such that:
	\begin{enumerate}
		\item $x_{i} + H \neq 0$, and
		\item $x_{i} + H \neq x_{j}+H$ for all $j\in A_1\cup \cdots \cup A_{k-1}$.
	\end{enumerate}
	If $B_k \neq \emptyset$ then we define $A_{k}=A_{k-1}\cup \{i_k\}$ where $i_k$ is the smallest integer in $B_k$,
	and if $B_k=\emptyset$ we define $A_k = A_{k-1}$.
	
	Thus, for all $k\in \N\setminus \{0\}$ we have defined $A_k$ and we have the sequence $(A_{k})_{k=1}^\infty$.
	This sequence is eventually constant because $[d]$ is finite.
	We define $A$ to be the infinitely repeating term of the sequence $(A_{k})_{k=1}^\infty$.
	Hence, we can write $A=\{i_1,\dots, i_n\}\subseteq [d]$.
	The first three claims follow from the inductive definition of $A$.
	
	The fourth claim is now forced because $S$ is a tree ring,
	otherwise we have a contradiction with the cycle $\{1+H,x_{i_1}+H, x_{i_1}x_{i_2}+H,x_{i_2}+H\}$.
	The fifth claim now follows immediately from the fourth claim.
	The sixth claim is also forced because $S$ is a tree ring,
	otherwise we have two different paths between the vertices $1+H$ and $x_i^p+H=x_j^q+H$.
\end{proof}

Bezrukov in \cite{BezrukovSpiders} studied powers of Macaulay posets whose Hasse graphs are trees.
He found that the Macaulay property among upper semilattices forces the posets to be spider posets.

\begin{dfn}[Upper Semilattice]\label{Upper_Semilattice_Definition}
	Suppose that we have a poset $\mathscr{P}$.
	For $a,b,s\in \mathscr{P}$ we say that $s$ is a \textit{supremum} of $a$ and $b$ if $a,b\leq s$,
	and if for all $c\in \mathscr{P}$ with $a\leq c$ and $b\leq c$ we have $s\leq c$.
	We say that $\mathscr{P}$ is an \textit{upper semilattice} if for any $a,b\in \mathscr{P}$ a supremum exists and is unique.
\end{dfn}

\begin{thm}[Bezrukov \cite{BezrukovSpiders} 1998]\label{Bezrukov_Upper_Semilattice}
	Suppose that $\mathscr{P}$ is a finite ranked upper semilattice and $n\geq r(\mathscr{P}) + 3$.
	If $\mathscr{P}^{\times, n}$ is Macaulay then $\mathscr{P}$ is isomorphic to a spider poset. 
\end{thm}

Combining everything we have developed so far, we get the following classification result.

\begin{thm}[Tree Ring Classification]\label{treeRingClassification}
	Suppose that $S$ is a level linearly independent tree ring such that $\mathscr{M}_S$ is finite and let $n\geq r(\mathscr{M}_S) + 3$.
	Then $S^{\otimes, n}$ is Macaulay iff $\mathscr{M}_S$ is isomorphic to the poset of monomials of a basic Bezrukov--Elsässer ring. 
\end{thm}
\begin{proof}
	By Theorem \ref{The_Cartesian_and_Tensor_Correspondence} we have that $S^{\otimes, n}$ is level linearly independent.
	Notice that Lemma \ref{Tree_Rings_Concrete} gives us that there is always a monomial order on $\mathscr{M}_S$,
	whence by Theorem \ref{The_Cartesian_and_Tensor_Correspondence} we always have a monomial order on $\mathscr{M}_{S^{\otimes, n}}$.
	Also, by using Theorem \ref{The_Cartesian_and_Tensor_Correspondence} again we have that $\mathscr{M}_{S^{\otimes, n}} = \mathscr{M}_{S}^{\times, n}$,
	and by Proposition \ref{Properties_of_Duals} we have $(\mathscr{M}_{S}^{\times, n})^\ast = (\mathscr{M}_{S}^\ast)^{\times, n}$.
	Therefore, $\mathscr{M}_{S^{\otimes, n}}^\ast = (\mathscr{M}_{S}^{\times, n})^\ast =(\mathscr{M}_{S}^\ast)^{\times, n}$.
	
	We deal with the forward direction first, so suppose $S^{\otimes, n}$ is Macaulay.
	Thus, by the Macaulay Correspondence Theorem \ref{Macaulay_Correspondence_Theorem} we have that $\mathscr{M}_{S^{\otimes, n}}$ is Macaulay.
	Hence, Bezrukov's Dual Lemma \ref{Bezrukov_Dual_Lemma} gives us that $\mathscr{M}_{S^{\otimes, n}}^\ast$ is Macaulay.
	Furthermore, by Lemma \ref{Tree_Rings_Concrete} we have that $\mathscr{M}_{S}^\ast$ is an upper semilattice.
	So, $\mathscr{M}_{S}^\ast$ is isomorphic to a spider poset by Theorem \ref{Bezrukov_Upper_Semilattice}.
	Therefore, $\mathscr{M}_{S}$ is isomorphic to the poset of monomials of a basic Bezrukov--Elsässer ring.
	
	The backwards direction follows from Corollary \ref{Dual_Spider_Macaulay_Theorem} and the Macaulay Correspondence Theorem \ref{Macaulay_Correspondence_Theorem}.
\end{proof} \subsection{Macaulay Rings that are not Tree Rings}\label{submatrix}

Chong's theorems \ref{Chong_Macaulay_1} and \ref{Chong_Macaulay_2} arose from the study of tensor products that involve quotients of polynomial rings by an ideal generated from some level.
Chong called these structures colored quotient rings as they were a generalization of the colored squarefree rings that Mermin and Murai studied.
This idea has been studied in Macaulay poset theory as well.

\begin{dfn}[Leck Rings]
	A \textit{basic Leck} ring has the form
	\begin{align*}
		\frac{K[x_1,\dots, x_d]}{\Pow(2,\dots, 2) + (x_1x_2\cdots x_d)}.
	\end{align*}
	A \textit{Leck} ring is the tensor product of basic Leck rings together with one Kruskal--Katona Ring.
\end{dfn}

Notice that when $d=2$, the poset of monomials of a basic Leck ring is just the dual of a star.
Thus, the following theorems provide another direction in generalizing the Mermin--Murai Theorem \ref{MerminMurai}.

\begin{thm}[Leck \cite{LeckUwe2001Osai, LeckUwe2002AGoL} 2001,2002]\label{Leck_Posets}
	The poset of monomials of a Leck ring is Macaulay.
\end{thm}

This time we don't need to deal with dual posets.
Everything is setup the right way and we get a corollary from the Macaulay Correspondence Theorem \ref{Macaulay_Correspondence_Theorem}.

\begin{cor}\label{Leck_Ring_Macaulay}
	Every Leck ring is Macaulay.
\end{cor}

Leck's results are very interesting for two reasons.
The first one is that the product does not involve any trees in many cases.
The second one is that the Macaulay order is not a domination order and not a block order.  \section{Quotients by Binomial Ideals}\label{Quotients_by_Binomial_Ideals}
\numberwithin{thm}{subsection}

The research around Macaulay rings has been very focused on rings that are quotients by a monomial ideal.
Toric ideals have also been considered \cite{GasharovVesselin2008Hfot, GasharovVesselin2011Hsam,MURAISATOSHI2011Frol}.
As far as the authors are aware, 
there are no published results that show that quotients by ideals that are not monomial and not toric can be Macaulay.
By using all the results in Section \ref{Translating_Between_Posets_and_Rings} and some solutions to discrete extremal problems,
we give two families of Macaulay rings that fall outside these previously studied categories.

\subsection{The Discrete Even Torus}\label{Tori_Posets}

\begin{dfn}(Karakhanyan--Riordan Rings)
	A \textit{basic Karakhanyan--Riordan} ring has the from
	\begin{align*}
		\frac{K[x_1, x_2]}{H}= 
		\frac{K[x_1, x_2]}{\Pow(p,p) + \DVP(2) + (x_1^{p-1}-x_2^{p-1})} = \frac{K[x_1, x_2]}{(x_1^p,x_2^p,x_1x_2, x_1^{p-1}-x_2^{p-1})}.
	\end{align*}
	We will denote the poset of monomials of a basic Karakhanyan--Riordan ring by $T(p)$.
	From here, a \textit{Karakhanyan--Riordan} ring is a tensor product of basic Karakhanyan--Riordan rings.
	The poset of monomials of a Karakhanyan--Riordan ring will be denoted by $T(k_1,\dots, k_n)$.
\end{dfn}

Karakhanyan--Riordan rings are similar to some Bezrukov--Elsässer rings.
In particular, the Hasse graph of a Karakhanyan--Riordan ring can be obtained from the Hasse graph of Bezrukov-Elsässer ring by gluing the top elements.
Some examples of Hasse graphs of basic Karakhanyan--Riordan rings can be seen in Figure \ref{Hasse_Graphs_of_some_basic_Karakhanyan_Riordan_Rings}.

\begin{figure}
	\centering
	\begin{subfigure}[t]{0.1\textwidth}
		\includegraphics[width=\textwidth]{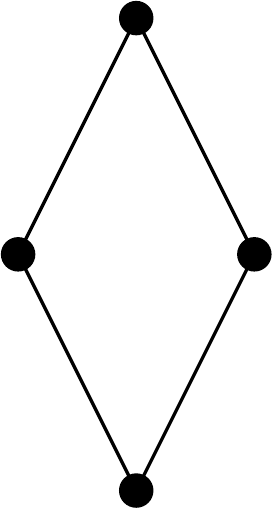}
	\end{subfigure}
	\hfill
	\begin{subfigure}[t]{0.1\textwidth}
		\includegraphics[width=\textwidth]{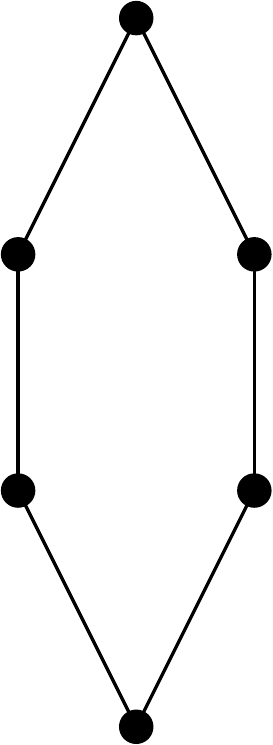}
	\end{subfigure}
	\hfill
	\begin{subfigure}[t]{0.1\textwidth}
		\includegraphics[width=\textwidth]{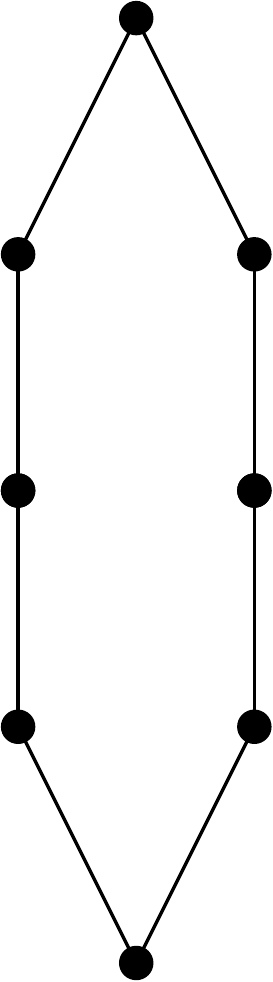}
	\end{subfigure}
	\hfill
	\begin{subfigure}[t]{0.1\textwidth}
		\includegraphics[width=\textwidth]{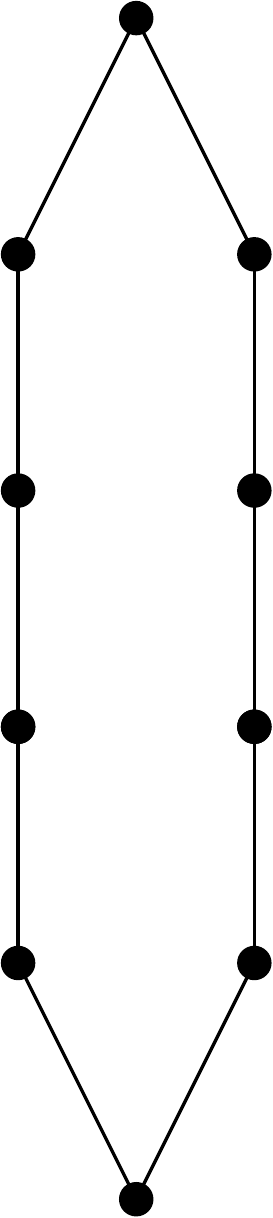}
	\end{subfigure}
	\hfill
	\caption{Hasse graphs of some basic Karakhanyan--Riordan rings.}\label{Hasse_Graphs_of_some_basic_Karakhanyan_Riordan_Rings}
\end{figure}

Hasse graphs of basic Karakhanyan--Riordan rings are even cycles.
These graphs have received a lot of attention in the area of vertex isoperimetric problems.
In many cases, the solution to a vertex isoperimetric problem implies the solution to a Macaulay problem.
Informally, 
the vertex isoperimetric problem is concerned with finding sets of vertices in simple graphs with minimum boundary.
We refer the reader to \cite{HarperBook} for an introduction to vertex isoperimetric problems.

The vertex isoperimetric problem on $T(k_1,\dots, k_n)$ was first posed by Wang and Wang in \cite{WangDa-Lun1977ECoa}.
The problem was completely settled by Karakhanyan in \cite{KarakhanyanTorus}.
Later, the case $T(k,\dots, k)$ was independently settled by Bollobás and Leader in \cite{BelaImreTorus}.
Building on the work of Bollobás and Leader,
Riordan in \cite{RiordanOliver1998AOot} independently discovered the general case that Karakhanyan figured out.
It known in the literature that many vertex isoperimetric problems imply the Macaulay poset property,
and the following result was observed by Bezrukov and Leck in \cite{Bezrukov2004}.

\begin{thm}[Karakhanyan--Riordan Macaulay Theorem]\label{Karakhanyan-Riordan_Theorem}
	$T(k_1,\dots, k_n)$ is Macaulay.
\end{thm}

Of course, we want to get a result on Macaulay rings from the above result.
We need to be a little bit more delicate now, 
since the ideal involved in the quotient is not monomial.

\begin{cor}\label{Tori_Rings_Macaulay}
	All Karakhanyan--Riordan rings are Macaulay.
\end{cor}
\begin{proof}
	First, 
	it is easy to see that the poset of monomials of a basic Karakhanyan--Riordan ring is level linearly independent,
	since the only concern is with the relation $x_1^{p-1} - x_2^{p-1}$ and this relation occurs on a level that has one element.
	Second there is a monomial order on every basic Karakhanyan--Riordan ring.
	We order every monomial that is not on the top level by the lexicographic order and we set the element on the top level to be the last element in our new monomial order.
	Thus, Theorem \ref{The_Cartesian_and_Tensor_Correspondence} implies that every Karakhanyan--Riordan ring is level linearly independent and there exists a monomial order on it.
	So, the conditions for the Macaulay Correspondence Theorem \ref{Macaulay_Correspondence_Theorem} are satisfied,
	whence Theorem \ref{Karakhanyan-Riordan_Theorem} implies that every Karakhanyan--Riordan ring is Macaulay.
\end{proof}

We just need to describe a Macaulay order on $T(k_1,\dots, k_n)$.
The Macaulay order we present here is given by Bezrukov and Leck in \cite{BezrukovSergeiL.2009ASPo},
and we reformulate it in terms of block orders.
Bezrukov and Leck gave a simpler proof of the Karakhanyan--Riordan Theorem by using a vertex isoperimetric inequality on finite grids by Bollob\'{a}s and Leader \cite{BollobásBéla1991Iiaf}.
Assume that $2\leq k_1 \leq \cdots \leq k_n$.
We are going to form a block order made up of domination orders.
Consider the ordered partitions
\begin{align*}
	T(k_i) = \{1+H < x_1+H < \cdots < x_1^{k_i-2}+H\} \cup \{x_2+H < x_2^2+H < \cdots < x_1^{k_i-1}+H=x_2^{k_i-1}+H\}.
\end{align*}
The starts of $T(k_1,\dots, k_n)$ are ordered by the colexicographic order,
and each block is ordered by the lexicographic order.
The dual of this block order is a Macaulay order on $T(k_1,\dots, k_n)$. \subsection{The Diamond Poset}\label{Diamond_Poset}

\begin{dfn}[Bezrukov--Piotrowski--Pfaff Rings]
	The \textit{basic Bezrukov--Piotrowski--Pfaff} ring is
	\begin{align*}
		\frac{K[x_1,x_2,x_3]}{\Pow(3,3,3) + \DVP(2) + (x_1-x_2, x_2-x_3)}
		= \frac{K[x_1,x_2,x_3]}{(x_1^3, x_2^3, x_3^3, x_1x_2, x_1x_3, x_2x_3,x_1^2-x_2^2, x_2^2-x_3^2)}.
	\end{align*}
	A \textit{Bezrukov--Piotrowski--Pfaff} ring is a tensor power of the above basic ring.
\end{dfn}

The Hasse graph of the basic Bezrukov--Piotrowski--Pfaff is called the \textit{diamond poset},
and can be seen in Figure \ref{Hasse_Graphs_of_Diamond}.

\begin{figure}
	\centering
	\begin{subfigure}[t]{0.3\textwidth}
		\includegraphics[width=\textwidth]{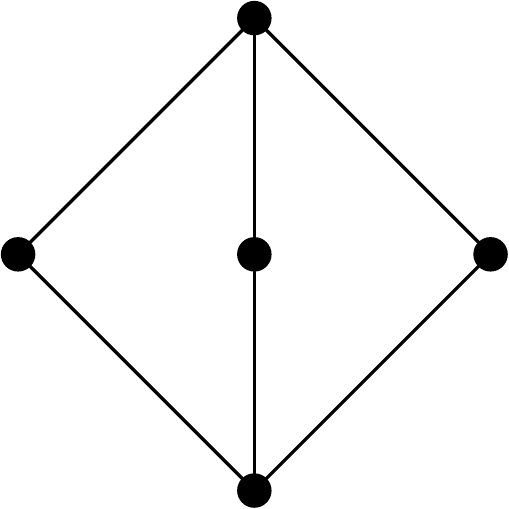}
	\end{subfigure}
	\caption{Hasse graph of the basic Bezrukov--Piotrowski--Pfaff ring.}\label{Hasse_Graphs_of_Diamond}
\end{figure}

\begin{thm}[Bezrukov--Piotrowski--Pfaff \cite{Diamond} 2004]
	All powers of the diamond poset are Macaulay.
\end{thm}

The following corollary follows right away and its proof is identical to the proof of Corollary \ref{Tori_Rings_Macaulay}.
Note that we again need to worry about level linear independence and the existence of a monomial order,
but it is easy to see that both of these requirements are satisfied.

\begin{cor}\label{Diamond_Macaulay}
	All Bezrukov--Piotrowski--Pfaff rings are Macaulay.
\end{cor}

The Macaulay order is again given by a block order.
This order can by found in \cite{Diamond} in an equivalent form.
Form the ordered partition
\begin{align*}
	\{1+H < x_1+H\} \cup \{x_2+H\} \cup \{x_3+H < x_1^2+H=x_2^2+H=x_3^2+H\}.
\end{align*}
We order the starts by the lexicographic order and each block by using the colexicographic order.
The resulting block order is a Macaulay order on the Cartesian powers of the diamond poset.  \numberwithin{thm}{section}
\section{Open Problems}\label{Open_Problems}

We state several open problems in this section.
Of course, we have the main topics of this paper.

\begin{prb}
	Find classes of Macaulay posets.
\end{prb}

\begin{prb}
	Find classes of Macaulay rings.
\end{prb}

However, let us give more detailed problems.
One of the things that makes the Macaulay Correspondence Theorem \ref{Macaulay_Correspondence_Theorem} work is the existence of monomial orders.

\begin{prb}
	Find classes of rings for which monomial orders exist.
\end{prb}

The other things that is very important for the Macaulay Correspondence Theorem \ref{Macaulay_Correspondence_Theorem} is level linear independence.

\begin{prb}
	Find classes of level linearly independent rings.
\end{prb}

The Macaulay problem on star posets and squarefree rings is completely settled,
but there is still work to be done when it comes to tree rings.
The classification in Theorem \ref{treeRingClassification} does not handle the cases for small powers of tree rings.
It is an easy exercise to check that the following power of a tree ring is Macaulay
\begin{align*}
	\left( \frac{K[x_1,x_2]}{(x_1^2,x_2^3, x_1x_2)} \right)^{\otimes, 2}.
\end{align*}
Of course, Theorem \ref{treeRingClassification} does not handle the case when we can have different tree rings in the product.
Bezrukov and Leck in \cite{Bezrukov2004} conjecture that one can have spider posets with the same leg length, 
but different number of legs,
and the product will still be Macaulay.
The same conjecture is stated by Harper in \cite{HarperBook}.
Chong's Theorem \ref{Chong_Macaulay_1} shows that infinite trees can appear in the product.

\begin{prb}
	Completely classify the posets of monomials for all Macaulay rings that are tensor products of tree rings.
	If possible, do this with a unified approach by using block orders that involve the hyperrectangle chaser order on the starts and domination orders on the blocks.
	Give a general Macaulay theorem that encompasses as special cases the original Macaulay Theorem \ref{Macaulay_1927},
	the Clements--Lindström Theorem \ref{Clements_Lindstrom_Theorem},
	the star Macaulay Theorem \ref{starTheorem} and the Mermin--Murai Theorem \ref{MerminMurai},
	the Bezrukov--Elsässer Theorem \ref{Spider_Macaulay_Theorem},
	Chong's Theorem \ref{Chong_Macaulay_1},
	and the Tree Ring Classification Theorem \ref{treeRingClassification}. 
\end{prb}

The Mermin--Peeva and Shakin Theorem \ref{Mermin_Peeva_and_Shakin} gives a nice answer to the problem of Bezrukov and Leck.
But Bezrukov and Leck wanted a more general statement that involves a chain of any length.
This suggests a very natural generalization of the Mermin--Peeva and Shakin Theorem,
note that the case $n=\infty$ in Conjecture \ref{Bezrukov_Leck_Chain_Conjecture} is Theorem \ref{Mermin_Peeva_and_Shakin}.

\begin{con}\label{Bezrukov_Leck_Chain_Conjecture}
	If $S$ is a Macaulay ring then $S\otimes K[x]/(x^n)$ is Macaulay.
\end{con}

Here is a slightly weaker version.

\begin{con}
	Suppose that $S$ is a quotient by a monomial ideal.
	If $S$ is Macaulay with a domination order then $S\otimes K[x]/(x^n)$ is Macaulay.
\end{con}

In Section \ref{Quotients_by_Binomial_Ideals} the quotients by binomial ideals that we considered have something in common.
The posets of monomials are obtained by taking chains of the same length and joining the bottom and top elements.
Here is a problem in this direction.

\begin{prb}
	Let $R=K[x_1,\dots, x_d]$ and put
	\begin{align*}
		H = \Pow(n,\dots, n) + \DVP(2) + (x_1^{n-1}-x_2^{n-1},x_2^{n-1}-x_3^{n-1},\dots, x_{d-1}^{n-1}- x_d^{n-1}).
	\end{align*}
	Find all cases when $R/H$ is Macaulay.
\end{prb}

In Section \ref{submatrix} we discussed Leck's results which show that we can have a tensor product that is Macaulay, 
but none of the individual rings in the product are tree rings.
The products in Leck's results concern posets of that form $\mathscr{M}_{[d]}(2,\dots, 2)$,
but with the top element removed.
In the case $d=2$, this is just the dual of a basic star poset.
Chong's results in Section \ref{Products_With_Multiset_Lattices} are obtained from a similar point of view.
Chong considered, 
after translating from rings to posets, 
products that are made up of $\mathscr{M}_{[d]}(\infty, \dots, \infty)$ but we also remove everything above a certain level.
This leads to asking about a natural generalization.

\begin{prb}\label{Leck_Chong}
	Find all Macaulay rings that are tensor products of rings of the form
	\begin{align*}
		\frac{K[x_1,\dots, x_d]}{\Pow(\ell_1,\dots, \ell_d) + (\Lvl_i)},
	\end{align*}
	for some $\ell_1,\dots, \ell_d\in \N \cup \{\infty\}$ and $i\in \N$.
\end{prb}

Of course, Leck's results concern the case $\ell_1=\dots = \ell_d = 2$,
and Chong's results concern the case $\ell_1=\dots = \ell_d = \infty$.
Solving Problem \ref{Leck_Chong} might lead to finding new types of orders, 
as Leck's results do not involve a domination order or a block order.
Both Leck and Chong discovered results that say that there must be very strict restrictions on $i$ for the Macaulay property to hold.
We refer the reader to \cite{ChongKaiFongErnest2015Hfoc, LeckUwe2001Osai, LeckUwe2002AGoL, LeckUwe2003Otop}.  
\section{Acknowledgments}
The author would like to thank Alexandra Seceleanu for all the support she provided.
It would have been impossible to get this done on time, if at all, without her help.
Steven J. Rosenberg for the numerous corrections of an early draft.
The author would like to thank the anonymous referee for many corrections and suggestions.


\bibliographystyle{acm}
\bibliography{./references}

\end{document}